\documentclass[a4paper,11pt]{article}
\usepackage{latexsym,amssymb,amsmath,amsthm,amsxtra,xypic}
\usepackage{stmaryrd}
\usepackage{amsfonts}
\usepackage{amscd}
\usepackage{mathrsfs}
\usepackage[dvips]{graphicx}
\usepackage[all]{xy}

\newcommand{\allowpagebreak}
\allowdisplaybreaks

\newtheorem{thm}{Theorem}[section]

\newtheorem{pro}[thm]{Proposition}
\newtheorem{ex}[thm]{Example}

\theoremstyle{definition}
\newtheorem{defi}[thm]{Definition}
\newtheorem{rmk}[thm]{Remark}

\newcommand {\emptycomment}[1]{}

\newcommand{\pf}{\noindent{\bf Proof.}\ }

\def\LM{\mathcal{LM}}

\newcommand{\hlie}{\h_{\mathrm{Lie}}}

\newcommand{\lam}{\lambda}
\newcommand{\si}{\sigma}

\newcommand{\gl}{\mathfrak {gl}}

\newcommand{\g}{\mathfrak g}

\newcommand{\wg}{\widehat{\mathfrak g}}

\newcommand{\wrho}{\widehat{\rho}}

\newcommand{\hg}{\widehat{\mathfrak g}}
\newcommand{\tg}{\widetilde{\mathfrak g}}
\newcommand{\h}{\mathfrak{h}}
\newcommand{\frkg}{\mathfrak g}
\newcommand{\frkh}{V}

\newcommand{\dlam}{_\lambda}

\newcommand{\dM}{f}
\newcommand{\E}{\mathrm{E}}
\renewcommand{\L}{L}

\newcommand{\Hom}{\mathrm{Hom}}

\newcommand{\Ker}{\mathrm{Ker}}

\newcommand{\End}{\mathrm{End}}

\newcommand{\ad}{\mathrm{ad}}

\newcommand{\id}{\mathrm{id}}

\newcommand{\sumn}{{\sum_{i=1}^n}\,}
\newcommand{\sumnoi}{{\sum_{i=1}^{n-1}}\,}
\newcommand{\sumnoii}{{\sum_{i=1}^{n-2}}\,}

\newcommand{\sumnoj}{{\sum_{j=1}^{n-1}}\,}
\newcommand{\sumnojj}{{\sum_{j=1}^{n-2}}\,}

\begin{document}

\allowdisplaybreaks

\title{Cohomology of $n$-Lie algebras in Loday-Pirashvili category}

\author{Tao Zhang\thanks{{E-mail:} zhangtao@htu.edu.cn}\\
{\footnotesize College of Mathematics and Information Science, Henan Normal University, Xinxiang 453007, PR China}}

\date{}
\maketitle



\begin{abstract}
In this article, we present the concept of $n$-Lie algebras in the Loday-Pirashvili category. We examine the representation and cohomology theory of $n$-Lie algebras in this category, and also explore their applications in infinitesimal deformation and abelian extension theory.
\end{abstract}


\section{Introduction}

Filippov introduced the concept of $n$-Lie algebras in \cite{Fil1}, which may also be referred to as Filippov algebras or Nambu-Lie algebras. An $n$-Lie algebra, denoted as $\frkg$, is a vector space equipped with an $n$-ary totally skew-symmetric linear map ($n$-bracket)
$\frkg\times\cdots\times\frkg\to \frkg: (x_1,\dots, x_n)\to [x_1,\dots, x_n]$  satisfying the following fundamental identity:
\begin{eqnarray}\label{eq:Jacobi-11}
[x_1, \dots, x_{n-1},[y_1,y_2, \dots , y_n]]=\sumn [y_1, \dots,[x_1, \dots, x_{n-1},y_i],\dots y_n],
\end{eqnarray}
for all $x_i, y_i \in \frkg$.
It dates back to Nambu's work  \cite{Nam} to generalize the classical Hamiltonian mechanics when $n=3$:
$$\{h_1,h_2,\{x_1, x_2, x_3,z\}\}=\{\{h_1,h_2,x\},y,z\}+\{x,\{h_1,h_2,y\},z\}+\{x_1, x_2, x_3,\{h_1,h_2,z\}\},$$
where $h_1,h_2$ are Hamiltonian and $\{\cdot,\cdot,\cdot\}$ is a ternary product.
When the $n$-ary bracket in \eqref{eq:Jacobi-11} is not skew-symmetric, it is called a Leibniz $n$-algebra \cite{DT,Tak2,CKL,CLP}.

The study of $n$-Lie algebra is closely related to many fields in mathematics and mathematical physics.
For example, the foundations of the theory of Nambu-Poisson manifolds were developed by Takhtajan in \cite{Tak1}.
The application of $n$-Lie algebras to the study of the supersymmetric gauge theory of multiple M2-branes and D-branes  is given in \cite{HHM,PS}.
The investigation of 3-Lie algebras for integrable system is given in \cite{CWWZ}.

The algebraic theory of $n$-Lie algebras has been extensively studied by several authors.
In particular, the (co)homology theory for $n$-Lie algebras was introduced by Takhtajan in \cite{Tak2,DT}, Gautheron in \cite{Gau}, and Rotkiewicz in \cite{Rot}.  One compare to the homology theory introduced in \cite{GTV}.
This theory is specifically tailored to the study of formal differential geometry and formal deformations of Nambu structures. The representation and cohomology theory of Leibniz $n$-algebras is also investigated in \cite{CKL,CLP}. Additionally, for the special case when $n=3$, the research can be found in \cite{Bai,FF,LZ,Zhang0,Zhang1,Zhang2}.
Recently, non-abelian extensions of  $n$-Lie algebras and perfect  Leibniz $n$-algebras were studied in \cite{AB,CKL1}.

In \cite{LP}, Loday and Pirashvili introduced a tensor category $\LM$ of linear maps. Roughly speaking, $\LM$ consists of linear maps $f:V\to W$ as objects. A morphism between two objects $f:V\to W$ and $f':V' \to W'$ is a pair of linear maps $(\phi, \psi)$, where $\phi:V\to V'$ and $\psi: W\to W'$, such that $\psi\circ f=f'\circ \phi$. In the same paper, they also introduced the concept of Lie algebra $(M,\g,f)$ in $\LM$. This Lie algebra is a representation $\rho:\g\to \gl(M)$ of an ordinary Lie algebra $\g$ on a vector space $M$, such that $f: M\to \frkg$ is an equivariant map satisfying the equation: \begin{eqnarray} \label{cm000} f(\rho(x)(m))=[x, f(m)]_\g, \end{eqnarray}
for all $x\in \g$ and $m\in M$.
When $f$ is a surjective map, then the above equation \eqref{cm000} is equivalent to
\begin{eqnarray} \label{cm000}
 f(\rho(f(m))(n))=[f(m), f(n)]_\g,
  \end{eqnarray}
  for all $m, n\in M$. This is called an embedding tensor or tensor hierarchy algebra in \cite{kotov-strobl,lavau-p,Pal}.
 Kurdiani constructed a novel cohomology theory for Lie algebras in $\LM$ using the iterated cone constructions in \cite{Kur}, Rovi studied Lie algebroids in this category \cite{Rov}.

In this paper, we introduce the concept of $n$-Lie algebras in $\LM$.
 Let ${\g}$ be an $n$-Lie algebra, ${M}$ be a vector space, and $\rho$ a representation of ${\g}$ on ${M}$. We define an $n$-Lie algebra in $\LM$ as an equivariant map $f: M\to \frkg$ satisfying the following condition: \begin{eqnarray} \label{cm001} f(\rho(x_1,\dots,x_{n-1})(m))&=&[x_1,\dots, x_{n-1}, f(m)]_\g, \end{eqnarray} for all $x\in \g$ and $m\in M$, where $\rho:\wedge^{n-1}\g\to \gl(M)$ is a representation of $\g$. When $n=2$, we recover Loday and Pirashvili's Lie algebra $(M,\g,f)$ in $\LM$ mentioned above. Therefore, the results of this paper are an $n$-order generalization of the results in \cite{LP,Kur}. Additionally, we show that this new concept of $n$-Lie algebras in $\LM$ gives rise to Leibniz $n$-algebras. Consequently, we investigate the representation and cohomology theory for this type of $n$-Lie algebras. Moreover, this paper provides a detailed exposition of the construction and properties of this cohomology theory. In a key observation, we find that the cohomology theory constructed by Kurdiani in \cite{Kur} can be extended to the case of Leibniz algebras and $n$-Lie algebras in $\LM$. Finally, we study their infinitesimal deformations and abelian extension theory in detail as applications.

The organization of this paper is as follows. In Section 2, we review some basic facts about $n$-Lie algebras and Leibniz $n$-algebras. In Section 3, we introduce the concept of $n$-Lie algebras in $\LM$ and study its elementary properties. We investigate the representation and cohomology theory of $n$-Lie algebras in $\LM$. Low-dimensional cohomologies are given in detail in this section. In Section 4, we prove that equivalent classes of abelian extensions are in one-to-one correspondence with the elements of the second cohomology groups. In Section~5, we study infinitesimal deformations of $n$-Lie algebras in $\LM$. The notion of Nijenhuis operators is introduced to describe trivial deformations.

\section{Cohomology of $n$-Lie algebras}\label{sec:cohom}

First, we recall some facts about $n$-Lie algebras and Leibniz algebras, basic references are \cite{Gau,Fil1,Tak1,Tak2,Loday}.

\begin{defi}[\cite{Fil1}]
An $n$-Lie algebra is a vector space $\frkg$ with an $n$-ary totally skew-symmetric linear map ($n$-bracket)
$\wedge{}^n\frkg\to \frkg: (x_1,\dots, x_n)\to [x_1,\dots, x_n]$ satisfying the following fundamental identity:
\begin{eqnarray}\label{eq:Jacobi-n}
[x_1, \dots, x_{n-1},[y_1,y_2, \dots , y_n]]=\sumn [y_1, \dots,[x_1, \dots, x_{n-1},y_i],\dots y_n],
\end{eqnarray}
for all $x_i, y_i \in \frkg$.
If the $n$-bracket is not skew-symmetric, then it is called a Leibniz $n$-algebra.
\end{defi}

Let $\g$ and $\g'$ be two  $n$-Lie algebras.
An $n$-Lie algebra homomorphism is a linear map $\phi:\g\to\g'$  such that following condition holds
\begin{eqnarray}
   \phi([x_1, \dots, x_n])&=&[\phi(x_1), \dots, \phi(x_n)]'.
\end{eqnarray}

Denote by $x=(x_1, \dots,x_{n-1})$ and $\ad(x) y_i=[x_1, \dots,x_{n-1}, y_i]$, then the above equality \eqref{eq:Jacobi-n} can be rewritten in the form
\begin{eqnarray}\label{eq:Jacobi-3'}
\ad(x)[y_1, y_2,\dots,y_n]=\sum_{i=1}^n[y_1,\dots,\ad(x)y_i,\dots,y_n].
\end{eqnarray}

Put $\L\triangleq\wedge^{n-1} \frkg$.
The elements in $\L$ are called fundamental objects.
Define an operation on fundamental objects by
\begin{eqnarray}\label{eq:fundamental}
x\circ y=\sum_{i=1}^{n-1}(y_1,\dots,[x_1,\dots,x_{n-1},y_i],\dots,y_{n-1}).
\end{eqnarray}
It can be  proved as in \cite{DT} that $\L$ is a Leibniz algebra  satisfying the following Leibniz rule
$$x\circ (y\circ z)= (x\circ y)\circ z+y\circ (x\circ z),$$
and
$$\ad(x)\ad(y) w-\ad(y)\ad(x) w=\ad(x\circ y) w,$$ 
for all $x,y,z\in\L, w\in \frkg$, i.e. $\ad: \L \to \End(\frkg)$ is a homomorphism of Leibniz algebras.

Recall that for a Leibniz algebra ${L}$, a representation is a vector space ${M}$
together with two bilinear maps
$$[\cdot,\cdot]_L:{L}\otimes {M}\to {M} \,\,\, \text{and}\,\,\,  [\cdot,\cdot]_R: {M}\otimes {L}\to {M},$$
satisfying the following three axioms
\begin{itemize}
\item[$\bullet$] {\rm(LLM)}\quad  $[x\circ y, m]_L=[x, [y, m]_L]_L-[y,[x, m]_L]_L$,
\item[$\bullet$] {\rm(LML)}\quad  $[m, x\circ y]_R=[[m, x]_R, y]_R+[x, [m, y]_R]_L$,
\item[$\bullet$] {\rm(MLL)}\quad  $[m, x\circ y]_R=[x,[m, y]_R]_L-[[x,m]_L, y]_R.$
\end{itemize}
By (LML) and (MLL) we also have
\begin{itemize}
\item[$\bullet$] {\rm(MMM)}\quad  $[[m, x]_R, y]_R+[[x,m]_L, y]_R=0.$
\end{itemize}
In fact, assume (LLM), one of (LML),(MLL),(MMM) can be derived from the other two.

The cohomology of  Leibniz algebras was defined in  \cite{Loday} as follows.
Let ${L}$ be a Leibniz algebra, ${M}$ be a representation of ${L}$. We denote the cochain complex by
$$ CL^{p}(L,M)\triangleq\Hom(\otimes^p L,M).$$
Define the coboundary operator
$$d_{p-1}:CL^{p-1}(L,M)\to CL^{p}(L,M)$$
by
\begin{eqnarray*}
&&d_{p-1}\omega(x_1,x_2,\dots,x_p)\\
&=&\sum_{i=1}^{p-1}(-1)^{i+1}[x_i,\omega(x_1,\dots,\widehat{x_i},\dots,x_p)]_L+(-1)^p[\omega(x_1,\dots,x_{p-1}),x_p]_R\\
&&+\sum_{1\leq i<j\leq p}(-1)^{i}\omega(x_1,\dots,\widehat{x_i},\dots,x_{j-1},x_i\circ x_j,x_{j+1},\dots,x_p),
\end{eqnarray*}
for all $x_i\in \L$. It is proved that  $d\circ d=0$.
For more details, see \cite{Loday}.

\begin{defi} Let $\frkg$ be an $n$-Lie algebra and $V$ be a vector space. Then $(V, \rho)$ is called a representation
of $\frkg$ if the following two conditions (R1) and (R2)  are satisfied:
\begin{itemize}
\item[{\rm(R1)}]\quad $[\rho(x_1,\dots,x_{n-1}),\rho(y_1,\dots,y_{n-1})]=\rho((x_1,\dots,x_{n-1})\circ (y_1,\dots,y_{n-1}))$,
\item[{\rm(R2)}]\quad $\rho(x_1,\dots,x_{n-2},[y_1,\dots,y_{n}])=$
            $\sumn(-1)^{n-i}\rho(y_1,\dots,\widehat{y_i}\dots, y_{n})\rho(x_1,\dots,x_{n-2},y_i)$.
\end{itemize}
\end{defi}

\begin{rmk}
The condition (R1) is equivalent to the following equation
\begin{eqnarray}
\notag&&\rho(x_1,\dots,x_{n-1})\rho(y_1,\dots,y_{n-1})\\
&=&\sumnoi \rho(y_1,\dots, [x_1,\dots,x_{n-1},y_i],\dots,y_{n-1}))+\rho(y_1,\dots,y_{n-1})\rho(x_1,\dots,x_{n-1}).
\end{eqnarray}
When only this condition is satisfied,  we call $(V, \rho)$ a weak representation of $\frkg$.
\end{rmk}

By direct computation, it is easy to obtain the following result.
\begin{pro}
Given a representation $\rho$ of the $n$-Lie algebra $\g$ on the vector space $V$. Define a skew-symmetric $n$-bracket on $\g\oplus V$ by
\begin{eqnarray}
[x_1+ u_1, \dots, x_n+u_n]=[x_1,\dots,x_{n}]+ \sumn(-1)^{n-i} \rho (x_1,\dots, \widehat{x_i},\dots, x_n)(u_i),
\end{eqnarray}
then $(\g\oplus V, [,])$ is an $n$-Lie algebra, which is called the semidirect product of $n$-Lie algebra $\g$ with $V$.
We denote this $n$-Lie algebra by $\g\ltimes V$.
Define a non-skew-symmetric $n$-bracket on $\g\oplus V$ by
\begin{eqnarray}\label{hemisemi}
[x_1+u_1, \dots, x_n+u_n]_H=[x_1,\dots,x_{n}]+\rho (x_1,\dots, x_{n-1})(u_n),
\end{eqnarray}
then $(\g\oplus V, [,]_H)$  is a Leibniz $n$-algebra, which is called the hemisemidirect product of $\g$ with $V$.
We denote this Leibniz $n$-algebra by $\g\ltimes_H V$. When $n = 2$, this hemisemidirect product  coincides with the notion defined in \cite{KW}.
\end{pro}


For example, given an $n$-Lie algebra $\frkg$, there is a natural {adjoint representation} on itself. The corresponding representation
$\ad(x_1,\dots,x_{n-1})$ is given by
\begin{eqnarray*}
\ad(x_1,\dots,x_{n-1})x_n=[x_1,\dots,x_{n-1},x_n].
\end{eqnarray*}
Thus, we obtain  an $n$-Lie algebra $\frkg\ltimes \frkg$ by
\begin{eqnarray}
[x_1+ y_1, \dots, x_n+y_n]=[x_1,\dots,x_{n}]+ \sumn(-1)^{n-i}[x_1,\dots, \widehat{x_i},\dots, x_n,y_i]
\end{eqnarray}
and a Leibniz $n$-algebra $\frkg\ltimes_H \frkg$ by
\begin{eqnarray}
[x_1+y_1, \dots, x_n+y_n]_H=[x_1,\dots,x_{n}]+[x_1,\dots, x_{n-1}, y_n]
\end{eqnarray}
for all $x_i, y_i \in \frkg$.

Let $\frkg$ be an $n$-Lie algebra and $\rho$ be a representation  of  $\frkg$ on $\frkh$.
It can be proved similar as in \cite{CKL,CLP} that  there is a representation of $L=\wedge^{n-1}\g$ as Leibniz algebra on $\Hom(\g,V)$ under the following maps:
\begin{eqnarray*}
&&[\cdot,\cdot]_L:\L \otimes \Hom(\g,V)\to \Hom(\g,V),\\
&&[\cdot,\cdot]_R: \Hom(\g,V)\otimes \L \to \Hom(\g,V),
\end{eqnarray*}
by
\begin{eqnarray}
\label{eq:leibniz01}{[(x_1,\dots,x_{n-1}),\phi]_L}(x_n)&=&\rho(x_1,\dots,x_{n-1})\phi(x_n)-\phi([x_1,\dots,x_{n-1},x_n]),\\
\notag\label{eq:leibniz02}{[\phi,(x_1,\dots,x_{n-1})]_R}(x_n)&=&\phi([x_1,\dots,x_n])-\sum_{i=1}^n(-1)^{n-i}\rho(x_1,\dots, \widehat{x_i},\dots x_n)\phi(x_i),\\
\end{eqnarray}
for all $\phi\in \Hom(\g,V), x_i\in \frkg$.
This observation is important because it provides a new characterization of the representation of $n$-Lie algebras in relation to the representation of Leibniz algebras. We will now establish that the converse is also valid.

\begin{pro}\label{pro:rep}
Let $\frkg$ be an $n$-Lie algebra. Then $\Hom(\g,V)$ equipped with the above two maps
$[\cdot,\cdot]_L$ and $[\cdot,\cdot]_R$ is a representation of Leibniz algebra $\L=\wedge^{n-1}\g$
if and only if the conditions (R1) and (R2) are satisfied, i.e. $(V,\rho)$ is a representation of  $\frkg$.
\end{pro}

\begin{proof}
For $x=(x_1,\dots,x_{n-1}),\, y=(y_1,\dots,y_{n-1})\in\L, \, y_n\in \frkg$, we compute the equality
\begin{equation}\label{leb01}
[x\circ y, \phi]_L(y_n)=[x, [y, \phi]_L]_L(y_n)-[y,[x, \phi]_L]_L(y_n).
\end{equation}
By definition \eqref{eq:leibniz01}, the left-hand side of \eqref{leb01} is equal to
\begin{eqnarray*}
[x\circ y, \phi]_L(y_n)=\rho(x\circ y)\phi(y_n)-\phi(\ad(x\circ y)y_n),
\end{eqnarray*}
and the right-hand side  of \eqref{leb01} is equal to
\begin{eqnarray*}
&&[x, [y, \phi]_L]_L(y_n)-[y,[x, \phi]_L]_L(y_n)\\
&=&\rho(x)[y, \phi]_L(y_n)-[y, \phi]_L(\ad(x)y_n)-\rho(y)[x, \phi]_L(y_n)+[x, \phi]_L(\ad(y)y_n)\\
&=&\rho(x)\rho(y)\phi(y_n)-\rho(x)\phi(\ad(y)y_n)-\rho(y)\phi(\ad(x)y_n)+\phi(\ad(y)\ad(x)y_n)\\
&&-\rho(y)\rho(x)\phi(y_n)+\rho(y)\phi(\ad(x)y_n)+\rho(x)\phi(\ad(y)y_n)-\phi(\ad(x)\ad(y)y_n)\\
&=&\rho(x)\rho(y)\phi(y_n)+\phi(\ad(y)\ad(x)y_n)-\rho(y)\rho(x)\phi(y_n)-\phi(\ad(x)\ad(y)y_n)\\
&=&[\rho(x),\rho(y)]\phi(y_n)-\phi([\ad(x),\ad(y)]y_n).
\end{eqnarray*}
Since $\ad: \L \to \End(\frkg)$ is a homomorphism of Leibniz algebras,
thus (LLM) holds for $[\cdot,\cdot]_L$ if and only if (R1) is valid for $\rho$.

Next, we compute the equality
$$[[\phi, x]_R, y]_R(y_n)+[[x,\phi]_L, y]_R(y_n)=0.$$
By \eqref{eq:leibniz01} and \eqref{eq:leibniz02} we have
\begin{eqnarray*}
&&{[(x_1,\dots,x_{n-1}),\phi]_L(z)+[\phi,(x_1,\dots,x_{n-1})]_R}(z)\\
&=&-\sumnoi(-1)^{n-i}\rho(x_1,\dots,\widehat{x_i},\dots,x_{n-1},w)\phi(x_i),
\end{eqnarray*}
thus
\begin{eqnarray*}
&&{[(x_1,\dots,x_{n-1}),\phi]_L+[\phi,(x_1,\dots,x_{n-1})]_R}\\
&=&-\sumnoi(-1)^{n-i}\rho(x_1,\dots,\widehat{x_i},\dots,x_{n-1},\cdot)\phi(x_i),
\end{eqnarray*}
where we denote $\rho(x_1,\dots,\widehat{x_i},\dots,x_{n-1},\cdot)\phi(x_i):\frkg\to \frkh$ by
$$\rho(x_1,\dots,\widehat{x_i},\dots,x_{n-1},\cdot)\phi(x_i)(z)=\rho(x_1,\dots,\widehat{x_i},\dots,x_{n-1},w)\phi(x_i).$$
Now replace $x_i$ by $y_i$, $\phi$ by $-\sumnoi(-1)^{n-i}\rho(x_1,\dots,\widehat{x_i},\dots,x_{n-1},\cdot)\phi(x_i)$ in \eqref{eq:leibniz02}, then we obtain
\begin{eqnarray*}
&&[[\phi, x]_R+[x,\phi]_L, y]_R(y_n)\\
&=&-\sumnoi(-1)^{n-i}\rho(x_1,\dots,\widehat{x_i},\dots,x_{n-1},\cdot)\phi(x_i)([y_1,\dots,y_n])\\
&&+\sumn(-1)^{n-j}\rho(y_1,\dots,\widehat{y_j},\dots,y_n)\Big\{\sumnoi(-1)^{n-i}\rho(x_1,\dots,\widehat{x_i},\dots,x_{n-1},\cdot)\phi(x_i)(y_j)\Big\}\\
&=&-\sumnoi(-1)^{n-i}\rho(x_1,\dots,\widehat{x_i},\dots,x_{n-1},[y_1,\dots,y_n])\phi(x_i)\\
&&+\sumnoi(-1)^{n-i}(\sumn(-1)^{n-j}\rho(y_1,\dots,\widehat{y_j},\dots,y_n)\rho(x_1,\dots,\widehat{x_i},\dots,x_{n-1},y_j))\phi(x_i).
\end{eqnarray*}
Thus (MMM) is valid for $[\cdot,\cdot]_L$ and $[\cdot,\cdot]_R$ if and only if (R2) is valid for $\rho$.

At last, one can check that (LML) or (MLL) is valid for $[\cdot,\cdot]_L$ and $[\cdot,\cdot]_R$ if and only if (R1) and (R2) are valid for $\rho$.
\end{proof}

Next, we define the cohomology theory of an $n$-Lie algebra in virtue of the cohomology theory of Leibniz algebra as follows.
A $p$-cochain on $\mathfrak{g}$ with coefficients in a representation $(V, \rho)$ is a multilinear map $\omega: \wedge^{p(n-1)+1} \mathfrak{g} \to V$. The space of $p$-cochains is denoted as $C^p(\mathfrak{g}, V)$.

\begin{thm}[\cite{CKL,CLP}]\label{theorem-cohomology}
Let $\frkg$ be an $n$-Lie algebra and $(V,\rho)$ be a representation of $\frkg$.
Then there exists a cochain complex $\left\{C(\frkg,\frkh)=\bigoplus_{p\geq 0}C^p(\frkg,\frkh),\delta \right\}$
which is a subcomplex of the Leibniz algebra $L=\wedge^{n-1}\frkg$ with coefficients in $\Hom(\frkg,V)$
such that the coboundary operator is given by
\begin{eqnarray}
\notag&&\delta_{p}\omega\left(X_1, \ldots, X_p, z\right) \\
\notag&= & \sum_{i=1}^p(-1)^i \omega\left(X_1, \dots, \widehat{X_i}, \dots, X_p,\left[X_i, z\right]\right) \\
\notag&& +\sum_{i=1}^p(-1)^{i+1} \rho\left(X_i\right) \omega\left(X_1, \dots, \widehat{X_i}, \dots, X_p, z\right) \\
\notag&& +\sum_{i=1}^{n-1}(-1)^{n+p-i+1} \rho\left(x^p_1, x^p_2, \dots, \hat{x}^p_i, \dots, x^p_{n-1}, z\right) \omega\left(X_1, \dots, X_{p-1}, x^p_i\right)\\
\label{ncohom}&& +\sum_{1 \leq i<j}(-1)^i \omega\left(X_1, \dots, \widehat{X_i}, \dots, X_{j-1}, X_i \circ X_j, X_{j+1}, \dots, X_p, z\right)
\end{eqnarray}
for all $X_i=\left(x_i^1, x_i^2, \dots, x_i^{n-1}\right) \in \wedge^{n-1} \mathfrak{g}$ and $z \in \mathfrak{g}$
such that $\delta\circ \delta=0$.
Thus, we obtain the cohomology group
$H^{\bullet}({\g}, V)=Z^{\bullet}({\g}, V)/B^{\bullet}({\g}, V)$ where $Z^{\bullet}({\g}, V)$ is the space of cocycles
and $B^{\bullet}({\g}, V)$ is the space of coboundaries.
\end{thm}

\begin{proof}
For the convenience of the reader, we provide a detailed proof for the case of $n$-Lie algebras since the formula \eqref{ncohom} in Theorem \ref{theorem-cohomology} is not explicitly given in \cite{CKL,CLP}. Now, to define the cochain complex for an $n$-Lie algebra $\frkg$ with coefficients in ${V}$, we consider a subcomplex of the Leibniz algebra $\wedge^{n-1}\frkg$ with coefficients in $\Hom(\frkg,V)$:
\begin{eqnarray*}
C^{p+1}(\frkg,\frkh)&\triangleq&
\Hom\left(\left(\otimes^p(\wedge{}^{n-1}\frkg\right)\otimes\frkg,\frkh\right)\cong\Hom\left(\otimes^p(\wedge{}^{n-1}\frkg),\Hom(\g,V)\right)\\
&=&\Hom\left(\otimes^{p} L,\Hom(\g,V)\right)=CL^p(L,\Hom(\g,V)).
\end{eqnarray*}
By the definition of cohomology of Leibniz algebras, we get
$$d_{p-1}:CL^{p-1}(L,\Hom(\g,V))\to CL^{p}(L,\Hom(\g,V)),$$
\begin{eqnarray*}
&&\delta_{p}\omega(X_1,X_2,\dots,X_p,z)=d_{p-1}\omega(X_1,X_2,\dots,X_p)(z)\\
&=&\sum_{i=1}^{p-1}(-1)^{i+1}[X_i,\omega(X_1,\dots,\widehat{X_i},\dots,X_p)]_L(z)+(-1)^p[\omega(X_1,\dots,X_{p-1}),X_p]_R(z)\\
&&+\sum_{1\leq i<j\leq p}(-1)^{i}\omega(X_1,\dots,\widehat{X_i}, \dots, X_{j-1},X_i\circ X_j,X_{j+1}, \dots,X_p)(z)\\
&=&\sum_{i=1}^{p-1}(-1)^{i+1}\{\rho(X_i)\omega(X_1,\dots,\widehat{X_i},\dots,X_p,z)-\omega(X_1,\dots,\widehat{X_i},\dots,X_p,[X_i,z])\}\\
&&+(-1)^p\omega(X_1,\dots,X_{p-1},[X_p,z])\\
&&-(-1)^p\sumn(-1)^{n-i}\rho(x^p_1, x^p_2, \dots, \hat{x}^p_i, \dots, x^p_{n-1}, z)\omega(X_1,\dots,X_{p-1},x^i_p)\\
&&+\sum_{1\leq i<j\leq p}(-1)^{i}\omega(X_1,\dots,\widehat{X_i}, \dots, X_{j-1},X_i\circ X_j,X_{j+1}, \dots,X_p)(z)\\
&= &\sum_{i=1}^p(-1)^i \omega\left(X_1, \dots, \widehat{X_i}, \dots, X_p,\left[X_i, z\right]\right) \\
&& +\sum_{i=1}^p(-1)^{i+1} \rho\left(X_i\right) \omega\left(X_1, \dots, \widehat{X_i}, \dots, X_p, z\right) \\
&& +\sum_{i=1}^{n-1}(-1)^{n+p-i+1} \rho\left(x^p_1, x^p_2, \dots, \hat{x}^p_i, \dots, x^p_{n-1}, z\right) \omega\left(X_1, \dots, X_{p-1}, x^p_i\right)\\
&& + \sum_{1 \leq i<j}(-1)^i \omega\left(X_1, \dots, \widehat{X_i}, \dots, X_{j-1}, X_i \circ X_j, X_{j+1}, \dots, X_p, z\right)
\end{eqnarray*}
for all $X_i=\left(x^i_1, x^i_2, \dots, x^i_{n-1}\right)\in \L=\wedge{}^{n-1}\frkg,\ z\in \frkg$.
In other words, the cohomology of an $n$-Lie algebra $\frkg$ with coefficients in $\frkh$ is defined to be
the cohomology of Leibniz algebra $\L$ with coefficients in $\Hom(\g,V)$.
Thus, we obtain that the cohomology group $H^{\bullet}({\g}, V)=Z^{\bullet}({\g}, V)/B^{\bullet}({\g}, V)$ is well defined.
\end{proof}

\begin{rmk}
L. Takhtajan in \cite{Tak2,DT}, P. Gautheron in \cite{Gau}, M. Rotkiewicz in \cite{Rot}, and the authors in \cite{GTV} have previously studied the cohomology theory for $n$-Lie algebras, each from a different perspective. However, we believe that the approach presented in Theorem \ref{theorem-cohomology} is the most direct and elegant. Specifically, for 3-Lie algebras, this method has been thoroughly examined in \cite{Zhang1,Zhang2}.
\end{rmk}

According to the above definition, a 1-cochain is a map $\nu\in\Hom(\g,V)$,
a 2-cochain is a map
$\omega\in\Hom\left(\wedge^{n}\g,\frkh\right)\subseteq\Hom\left(\wedge{}^{n-1}\frkg,\Hom(\g,V)\right)$,
and the coboundary operator is given by
\begin{eqnarray}
\label{eq:cobound01}\delta_{1}\nu(X_1,z)&=&d_{0}\nu(X_1)(z)=-[\nu,X_1]_R(z),\\
\label{eq:cobound02}\delta_{2}\omega(X_1,X_2,z)&=&[X_1,\omega(X_2)]_L(z)+[\omega(X_1),X_2]_R(z)-\omega(X_1\circ X_2)(z).
\end{eqnarray}

Put $X_1=(x_1,\dots,x_{n-1})\in \L,\ z=x_n\in\frkg$ in the equality \eqref{eq:cobound01}, then by \eqref{eq:leibniz02} we get
\begin{eqnarray}\label{eq:1coc}
\delta_1\nu(x_1,\dots,x_n)&=&\sumn(-1)^{n-i}\rho(x_1,\dots,\widehat{x_i},\dots,x_n)\nu(x_i)-\nu([x_1, \dots, x_n]).
\end{eqnarray}
\begin{defi}
Let $\frkg$ be an $n$-Lie algebra and $(V, \rho)$ be an $\frkg$-module. Then a map $\nu\in\Hom(\g,V)$
is called 1-cocycle if and only if
\begin{eqnarray}\label{eq:0coc}
\nu([x_1, \dots, x_n])=\sumn(-1)^{n-i}\rho(x_1,\dots,\widehat{x_i},\dots,x_n)\nu(x_i),
\end{eqnarray}
and a map $\omega: \wedge^n\frkg\to V$ is called a 2-coboudary if there exists a map $\nu\in\Hom(\g,V)$ such that $\omega=\delta_1\nu$.
\end{defi}
Put $X_1=(x_1,\dots,x_{n-1})\in \L, X_2=(y_1,\dots,y_{n-1})\in \L$, $z=y_n\in\frkg$ in the equality \eqref{eq:cobound02}, then we get
\begin{eqnarray*}
&&\delta_{2}\omega(x_1,\dots,x_{n-1},y_1,\dots,y_{n-1},y_n)\\
&=&[x_1,\dots,x_{n-1},\omega(y_1,\dots,y_{n-1})]_L(y_n)+[\omega(x_1,\dots,x_{n-1}),y_1,\dots,y_{n-1}]_R(y_n)\\
&&-\omega((x_1,\dots,x_{n-1})\circ (y_1,\dots,y_{n-1}))(y_n).
\end{eqnarray*}
We compute the right-hand side of the above equation as follows:
\begin{eqnarray*}
&&[x_1,\dots,x_{n-1},\omega(y_1,\dots,y_{n-1})]_L(y_n)\\
&=&\rho(x_1,\dots,x_{n-1})\omega(y_1,\dots,y_{n-1},y_n)-\omega(y_1,\dots,y_{n-1},[x_1,\dots,x_{n-1},y_n]),\\[1em]
&&[\omega(x_1,\dots,x_{n-1}),y_1,\dots,y_{n-1}]_R(y_n)\\
&=&\omega(x_1,\dots,x_{n-1},[y_1,\dots,y_n])\\
&&-\sumn(-1)^{n-i}\rho(y_1,\dots,\widehat{y_i},\dots, y_3)\omega(x_1,\dots,x_{n-1},y_i),\\[1em]
&&\omega((x_1,\dots,x_{n-1})\circ (y_1,\dots,y_{n-1}))(y_n)\\
&=&\sum_{i=1}^{n-1}\omega(y_1,\dots,[x_1,\dots,x_{n-1},y_i],\dots,y_{n-1},y_n).
\end{eqnarray*}
Thus, we obtain the following Definition \ref{def:1coc}.
\begin{defi}\label{def:1coc}
Let $\frkg$ be an $n$-Lie algebra and $(V, \rho)$ be a representation. Then a map
$\omega: \wedge^n\frkg\to V$ is called a 2-cocycle if  $\forall x_i, y_i\in \frkg$,
\begin{eqnarray}\label{eq:1coc}
\nonumber&& \omega(x_1,\dots,x_{n-1},[y_1, \dots, y_n])+\rho(x_1,\dots,x_{n-1})\omega(y_1,\dots, y_n)\\
\nonumber&=&\sumn\omega(y_1,\dots,[x_1,\dots,x_{n-1}, y_i],\dots, y_n)\\
&&+\sumn(-1)^{n-i}\rho(y_1,\dots,\widehat{y_i},\dots, y_n)\omega(x_1,\dots,x_{n-1},y_i).
\end{eqnarray}
\end{defi}

\section{$n$-Lie algebras in $\LM$}
In this section, we will investigate the representation and cohomology theory of $n$-Lie algebras in $\LM$. We will begin by introducing the concept of $n$-Lie algebras in $\LM$, which recovers the concept of Lie algebras in $\LM$ studied by Loday and Pirashvili in \cite{Loday}. Furthermore, detailed information on low-dimensional cohomologies will be provided in this section.

\begin{defi}\label{defi:Lie 2}
An $n$-Lie algebra in $\LM$ consists of a linear map $\dM: M\to \frkg$, where $\g$ is an $n$-Lie algebra, $(M,\rho)$ is a representation of $\g$, and $f$ is an equivariant map satisfying the following condition:
\begin{eqnarray}\label{eq:01}
f(\rho(x_1,\dots,x_{n-1})(m))=[x_1, \dots, x_{n-1}, f(m)],
\end{eqnarray}
for all  $x\in \g$ and $m\in M$. We denote an $n$-Lie algebra in $\LM$ by $(M,\g, f)$.
\end{defi}

When $f$ is a surjective map, then the above equation \eqref{eq:01} is equivalent to
\begin{eqnarray} \label{cm000}
 f(\rho(f(m_1),\dots, f(m_{n-1}))(m_n))=[f(m_1),\dots, f(m_n)]_\g,
  \end{eqnarray}
  for all $m_i\in M$. This is called an embedding tensor of $n$-Lie algebras.

For any Leibniz $n$-algebra $\mathfrak{h}$, the quotient map $\pi:\mathfrak{h}\to\hlie$ is an $n$-Lie algebra in $\LM$, where
$\hlie$  is $\mathfrak{h}$ mod its Leibniz kernel
$$K=[\h,\dots, \h]=\operatorname{span}\{[x_1,x_2,\dots, x_n]|\exists x_i, x_j\in \mathfrak{h}, x_i=x_j\}.$$
Conversely, given an $n$-Lie algebra $(M,\g, f)$ in $\LM$, a Leibniz $n$-algebra can be obtained on $M$ using the following approach.

\begin{thm}\label{thm:main01}
For any $n$-Lie algebra  $(M,\g, f)$ in $\LM$, we obtain a Leibniz $n$-algebra on $M$  with the $n$-bracket given by:
\begin{eqnarray}\label{eq:02}
[m_1, \dots, m_n]\triangleq \rho(f(m_1),\dots, f(m_{n-1}))(m_n).
\end{eqnarray}
\end{thm}

\begin{proof}
First, we prove that $f$ is an algebraic homomorphism from $M$ to $\g$. In fact, by equation \eqref{eq:01} we get
$$f([m_1, \dots, m_n])=f([f(m_1),\dots, f(m_{n-1}), m_n])=[f(m_1), \dots, f(m_n)].$$

Next, we verify that the  $n$-bracket defined by \eqref{eq:02} on  $M$ satisfying the fundamental identity:
$$[m_1,\dots, m_{n-1},[p_1, \dots, p_n]]=\sumn [p_1, \dots [m_1,\dots, m_{n-1}, p_i]\dots, p_n]].$$
The left-hand side of the above equation is equal to
\begin{eqnarray*}
&&[m_1,\dots, m_{n-1},[p_1, \dots, p_{n-1}, p_n]]\\
&=&[f(m_1),\dots, f(m_{n-1}),[f(p_1), \dots, f(p_{n-1}), p_n]]\\
&=&\rho(f(m_1), \dots, f(m_{n-1}))\rho(f(p_1), \dots, f(p_{n-1}))(p_n),
\end{eqnarray*}
and the right-hand side   is equal to
\begin{eqnarray*}
&&\sumn [p_1, \dots, [m_1, \dots, m_{n-1},p_{i}], \dots, p_n]]\\
&=&\sumn [f(p_1), [f(m_1),\dots, f(m_{n-1}),f(p_i)], \dots, f(p_{n-1}), p_n]]\\
&=&\sumnoi \rho(f(p_1), [f(m_1),\dots, f(m_{n-1}),f(p_i)], \dots, f(p_{n-1}))(p_n)\\
&&+\rho(f(p_1), \dots, f(p_{n-1}))\rho(f(m_1), \dots, f(m_{n-1}))(p_n)\\
&=&\sumnoi \rho((f(m_1),\dots, f(m_{n-1}))\circ (f(p_1), \dots, f(p_{n-1}))(p_n)\\
&&+\rho(f(p_1), \dots, f(p_{n-1}))\rho(f(m_1), \dots, f(m_{n-1}))(p_n),
\end{eqnarray*}
Thus, the two sides are equal to each other because $(M,\rho)$ is a representation of $\g$.
The proof is completed.
\end{proof}


Let $(M,\g, f)$ and $(M',\g', f')$ be two $n$-Lie algebras in $\LM$. A homomorphism between them is
$\phi=(\phi_0,\phi_1)$ where $\phi_0:\g\to\g'$ is an $n$-Lie algebra homomorphism and $\phi_1:M\to M'$ is an equivariant map such that
\begin{eqnarray}
   f'\circ \phi_1&=&\phi_0\circ f,\\
   \phi_0([x_1, \dots, x_n])&=&[\phi_0(x_1), \dots, \phi_0(x_n)]',\\
 \phi_1(\rho(x_1,\dots, x_{n-1})(m))&=&\rho'(\phi_0(x_1), \dots, \phi_0(x_{n-1}))(\phi_1(m)).
\end{eqnarray}

\begin{pro}\label{prop1}
Let $\g$ be an $n$-Lie algebra and $(M,\rho)$ be a representation.
Then  $ f: M\rightarrow\g$ is an $n$-Lie algebra in $\LM$
if and only if the map $\id+f : \g\rtimes M \to \g\rtimes\g$ is a homomorphism of $n$-Lie algebras.
\end{pro}

\begin{proof}
We will verify when the map $\id+f: \g\rtimes M \to \g\rtimes\g$ is a homomorphism of $n$-Lie algebras:
\begin{eqnarray}\label{prop-hom-01}
&&(\id+f)([x_1+ m_1, \dots, x_n+ m_n])={[(\id+f)(x_1+ m_1),\dots, (\id+f)(x_n+ m_n)]}.
\end{eqnarray}
On one hand,
\begin{eqnarray*}
&&(\id+f) [x_1+ m_1, \dots, x_n+ m_n]\\
&=&(\id+f)\Big([x_1,\dots,x_{n}]+ \sumn(-1)^{n-i} \rho (x_1,\dots, \widehat{x_i},\dots, x_n)(m_i)\Big) \\
&=&[x_1,\dots,x_{n}]+ \sumn(-1)^{n-i} f(\rho (x_1,\dots, \widehat{x_i},\dots, x_n)(m_i).
\end{eqnarray*}
On the other hand,
\begin{eqnarray*}
&&{[(\id+f)(x_1+ m_1),\dots, (\id+f)(x_n+ m_n)]}\\
&=&{[x_1+ f(m_1),\dots,x_n+ f(m_n)]}\\
&=&[x_1,\dots,x_{n}]+ \sumn(-1)^{n-i} [x_1,\dots, \widehat{x_i},\dots, x_n, f(m_i)].
\end{eqnarray*}
Thus, the two sides of above equation \eqref{prop-hom-01} are equal to each other if and only if
$$f(\rho(x_1,\dots,\widehat{x_i},\dots,x_{n})(m_i))=[x_1,\dots,\widehat{x_i},\dots,x_{n}, f(m_i)]$$
holds, which is the condition \eqref{eq:01}.
Therefore  $(\id+f)$ is a homomorphism of $n$-Lie algebras if and only if  the conditions \eqref{eq:01} hold.
This completes the proof.
\end{proof}

\begin{ex}
Consider the map $\pi:\g\ltimes_H V\to \g$ where $\pi$ is the natural projection, and the action of $\g$ on $\g\ltimes_H V$ given by
$$\rho(x_1,\dots, x_{n-1})(x+v)\triangleq[x_1,\dots, x_{n-1},x]+ \rho(x_1,\dots, x_{n-1})(v).$$
Then by Theorem \ref{thm:main01}, $\g\ltimes_H V$  becomes a Leibniz $n$-algebra under the $n$-bracket
\begin{eqnarray*}
&&[x_1+ u_1, \dots, x_n+u_n]_\pi\\
&=&\rho (\pi(x_1+ u_1), \dots, \pi(x_{n-1}+ u_{n-1}))(x_{n}+u_n)\\
&=&[x_1,\dots,x_{n}]+ \rho (x_1,\dots, x_{n-1})(u_n),
\end{eqnarray*}
which is exactly the $n$-bracket \eqref{hemisemi} in the hemisemidirect product $\g\ltimes_H V$.
\end{ex}

\begin{ex}
For an $n$-Lie algebra $\g$, we define $f=[\cdot,\dots,\cdot]:\wedge^n \g\to \g$ and the representation of $\g$  on $\wedge^n \g$ given by
$$\rho(x_1,\dots,x_{n-1})(y_1,\dots, y_n)\triangleq\sumn \left(y_1,\dots,[x_1,\dots,x_{n-1},y_i],\dots, y_n\right)$$
for all $x_i\in \g, (y_1,\dots, y_n)\in \wedge^n \g$.
In this case, by Theorem \ref{thm:main01}, $\wedge^n \g$  becomes a Leibniz $n$-algebra under the bracket
\begin{eqnarray*}
{[Y_1, \dots, Y_n]}&=&\rho(f(Y_1),\dots,f(Y_{n-1}))(Y_n)\\
&=&\sumn \left(y_1,\dots,[[y_1^1,\dots,y_{n}^1],[y_1^{n-1},\dots,y_{n}^{n-1}],y_i],\dots, y_n\right)
\end{eqnarray*}
 for all $Y_i=(y_1^i,\dots,y_n^i)\in \wedge^n \g$.
\end{ex}

\subsection{Representations}

\begin{defi}
Let $(M,\g, f)$ be an $n$-Lie algebra in $\LM$.
A representation  $\rho=(\rho_{1},\rho_{2},\rho_{3})$ of $(M,\g,f)$ is an object $(V,W,\varphi)$ in $\LM$ such that
the following conditions are satisfied:

(i) $(V,\rho_{1})$ and $(W,\rho_{2})$ are representations of $n$-Lie algebra $\g$;

(ii) $\varphi$ is an equivariant map with respect to $\rho_{1}$ and $\rho_{2}$:
\begin{eqnarray}\label{defrep02}
&&\varphi\circ\rho_1(x_1, \dots,x_{n-1})(v)=\rho_2(x_1, \dots,x_{n-1})\circ\varphi(v);
\end{eqnarray}

(iii) there exists a map $\rho_3:  \underbrace{\g\otimes\dots\otimes \g}_{n-2}  \otimes M\longrightarrow \Hom(W,V)$ such that:
\begin{eqnarray}\label{defrep03}
\varphi\circ\rho_3(x_1,\dots,x_{n-2},m)(w)=\rho_2(x_1,\dots,x_{n-2},f(m))(w),
\end{eqnarray}
where $x_1, \dots,x_{n-1}\in \g$, $m\in M$, $v\in V$ and  $w\in W$;

(iv) the following compatibility conditions hold:
\begin{eqnarray}
\notag&&\rho_1(x_1,\dots, x_{n-1})\rho_{3}(y_1,\dots,y_{n-2},m)(w)\\
\notag&=&\rho_3\big(y_1,\dots,y_{n-2},\rho(x_1,\dots, x_{n-1})(m)\big)(w)\\
\notag&&+\rho_3(y_1,\dots,y_{n-2},m))\rho_{2}(x_1,\dots, x_{n-1})(w)\\
\label{defrep04}&&+\sum_{i=1}^{n-2}\rho_3(y_1,\dots,[x_1,\dots, x_{n-1},y_i],\dots,y_{n-2},m)(w),\\[1em]
 \notag&&\rho_{3}(x_1, \dots, x_{n-2}, \rho(y_1,\dots, y_{n-1})(m))(w)\\
\notag&=& \rho_1(y_1, \dots, y_{n-1})\rho_{3}(x_1,\dots, x_{n-2},m)(w)\\
\label{defrep05}&&+\sumnoi(-1)^{n-i}\rho_{3}(y_1,\dots,\widehat{y_i} \dots, y_{n-1},m)\rho_2 (x_1, \dots, x_{n-2}, y_i)(w),
\end{eqnarray}
\begin{eqnarray}
\notag&&\rho_{3}(x_1,\dots, x_{n-3}, [y_1,\dots,y_{n}],m)(w)\\
\label{defrep06}&=&\sumn(-1)^{n-i}\rho_1(y_1,\dots,\widehat{y_i}\dots, y_{n})\rho_3(x_1,\dots, x_{n-3}, y_i, m)(w),\\[1em]
 \notag&&\rho_{3}(x_1, \dots, x_{n-2}, m)\rho(y_1,\dots, y_{n-1})(w)\\
\notag&=& \rho_1(y_1, \dots, y_{n-1})\rho_{3}(x_1,\dots, x_{n-2},m)(w)\\
\label{defrep07}&&+\sumnoi(-1)^{n-i}\rho_{3}\left(y_1,\dots,\widehat{y_i} \dots, y_{n-1},\rho(x_1, \dots, x_{n-2}, y_i)(m)\right)(w).
\end{eqnarray}
\end{defi}

From  the map $\rho_3$,  we define the $n$-linear map
$\{\cdot,\dots,\cdot\}: \underbrace{\g\otimes\dots\otimes \g}_{n-2}\otimes M\otimes  W\to V$ by
\begin{equation}
\{x_1,\dots,x_{n-2},m,w\}\triangleq\rho_3(x_1,\dots,x_{n-2},m)(w).
\end{equation}
Then the above equation \eqref{defrep03} is equivalent to
\begin{eqnarray}
&&\varphi(\{x_1,\dots,x_{n-2},m,w\})=\rho_2(x_1,\dots,x_{n-2},f(m))(w).
\end{eqnarray}
The above equations \eqref{defrep02} and  \eqref{defrep03} imply that the following two diagrams commute:
\begin{equation}
\label{module square 1}
\begin{gathered}
 \xymatrixcolsep{3.5pc}
 \xymatrixrowsep{2.5pc}
    \xymatrix{
 \g \otimes  \dots\otimes\g \otimes V  \ar[r]^-{\rho_1}    \ar[d]_{1 \otimes  \dots\otimes 1 \otimes \varphi} & V  \ar[d]^{\varphi} \\
       \g\otimes   \dots\otimes\g \otimes W  \ar[r]^-{\rho_2}  &  W  }
\end{gathered}\quad \mbox{and}\quad
\begin{gathered}
 \xymatrixcolsep{3.5pc}
 \xymatrixrowsep{2.5pc}
    \xymatrix{
\g \otimes\dots\otimes M\otimes  W  \ar[r]^-{\{\cdot,\dots,\cdot\}}    \ar[d]_{1_\g\otimes\dots\otimes f \otimes 1_W} & V  \ar[d]^{\varphi} \\
       \g\otimes \dots\otimes\g\otimes W  \ar[r]^-{\rho_2}  &  W  }
\end{gathered}
\end{equation}

For $(V,W,\varphi)=(M,\g, f)$, we get the adjoint representations of $n$-Lie algebras $(M,\g, f)$ on itself. This can be seen as following:

(i) $M$ and $\g$ are representations of $\g$ given by $\rho_1(x_1,  \dots,x_{n-1})(y)=\ad(x_1, \dots,x_{n-1})(y)$ and $\rho_2(x_1, \dots,x_{n-1})(m)=\rho(x_1, \dots,x_{n-1})(m)$;

(ii) $\varphi=f$ is an equivariant map:
\begin{eqnarray*}
f([x_1, \dots,x_{n-1}, m])=[x_1, \dots,x_{n-1}, f(m)];
\end{eqnarray*}

(iii) there exists a bilinear map $\{\cdot,\dots,\cdot\}: \g \otimes\dots\otimes M\otimes  W\to V$ defined by $$\{x_1,\dots,x_{n-2},m,x_{n-1}\}\triangleq-\rho(x_1,  \dots,x_{n-1}, m)$$
 such that
\begin{eqnarray*}
f(\{x_1,\dots,x_{n-2},m,x_{n-1}\})&=&-f(\rho(x_1,  \dots,x_{n-1}, m))\\
&=&-[x_1, \dots,x_{n-1},  f(m)]\\
&=&-[x_1, \dots, f(m), x_{n-1}]\\
&=&\rho_2(x_1,\dots,x_{n-2},f(m))(x_{n-1}).
\end{eqnarray*}

(iv) the compatibility conditions \eqref{defrep04}--\eqref{defrep07}  are satisfied by direct computations.

\medskip

We construct semidirect products of an $n$-Lie algebra structure in $\LM$ using its representations.
\begin{pro}\label{pro:semidirectproduct}
  Given a representation of the $n$-Lie algebra $(M,\g, f)$ on $(V,W,\varphi)$. Define on $(M\oplus V,\frkg\oplus W, \widehat{f})$ an $n$-Lie algebra in $\LM$ by
\begin{equation}\label{semi11}
\left\{\begin{array}{ll}
&\widehat{f}(m+v)\triangleq f(m)+\varphi(v),\\
&{[x_1+w_1, \dots, x_n+w_n]}\triangleq[x_1, \dots, x_n]+\sumn(-1)^{n-i} \rho_2 (x_1,\dots, \widehat{x_i},\dots, x_n)(w_i),\\
&{\wrho(x_1+w_1,\dots, x_{n-1}+w_{n-1})(m+v)}\triangleq \rho(x_1, \dots, x_{n-1})(m)+\rho_{1}(x_1,\dots, x_{n-1})(v)\\
&\qquad\qquad\qquad\qquad\qquad\qquad\qquad\qquad+\sumnoi (-1)^{n-i}\rho_{3}(x_1,\dots, \widehat{x_i},\dots,x_{n-1},m)(w_i).
\end{array}\right.
\end{equation}
Then $(M\oplus V,\frkg\oplus W, \widehat{f})$ is an $n$-Lie algebra in $\LM$, which is called the semidirect product of  $(M,\g, f)$ and $(V,W,\varphi)$.
\end{pro}

\begin{proof}
First, we verify that $\widehat{f}$ is an equivariant map:
\begin{eqnarray}\label{semi01}
\widehat{f}{\big({\wrho(x_1+w_1,\dots, x_{n-1}+w_{n-1})(m+v)}\big)}={[x_1+w_1, \dots, x_{n-1}+w_{n-1},\widehat{f}(m+v)]}.
\end{eqnarray}
The left-hand side of \eqref{semi01} is
\begin{eqnarray*}
&&\widehat{f}{\big({\wrho(x_1+w_1,\dots, x_{n-1}+w_{n-1})(m+v)}\big)}\\
&=&\widehat{f}\rho(x_1, \dots, x_{n-1})(m)+\widehat{f}\rho_{1}(x_1,\dots, x_{n-1})(v)\\
&&-\widehat{f}\rho_{3}(w_1,\dots, w_{n-1})(m)\\
&=&f(\rho(x_1, \dots, x_{n-1})(m))+\varphi\circ\rho_{1}(x_1, \dots, x_{n-1})(v)\\
&&+\sumnoi (-1)^{n-i}\varphi\circ\rho_{3}(x_1,\dots, \widehat{x_i},\dots,x_{n-1},m)(w_i),
\end{eqnarray*}
and the right-hand side of \eqref{semi01} is
\begin{eqnarray*}
&&{[[x_1+w_1, \dots, x_{n-1}+w_{n-1},\widehat{f}(m+v)]}\\
&=&[[x_1+w_1, \dots, x_{n-1}+w_{n-1}, f(m)+\varphi(v)]\\
&=&[x_1, \dots, x_{n-1},\dM(m)]+\rho_{2}(x_1, \dots, x_{n-1})(\varphi(v))\\
&&+\sumnoi(-1)^{n-i} \rho_2 (x_1,\dots, \widehat{x_i},\dots, f(m))(w_i).
\end{eqnarray*}
Thus, the two sides of  \eqref{semi01} are equal to each other by \eqref{defrep03}.

Next, we  prove that $\widehat{\rho}$ is a representation of $\g\oplus W$ on $M\oplus V$ satisfying (R1):
\begin{eqnarray}\label{semi02}
\notag&&\wrho(x_1+w_1,\dots, x_{n-1}+w_{n-1})\wrho(y_1+w'_1,\dots, y_{n-1}+w'_{n-1})(m+v)\\
&&-\wrho(y_1+w'_1,\dots,y_{n-1}+w'_{n-1})\wrho(x_1+w_1,\dots, x_{n-1}+w_{n-1})(m+v)\\
\notag&=&\sumnoi \wrho(y_1+w'_1,\dots, [x_1+w_1,\dots,x_{n-1}+w_{n-1},y_i+w'_i],\dots,y_{n-1}+w'_{n-1}))(m+v)
\end{eqnarray}

The left-hand side of \eqref{semi02} is
\begin{eqnarray*}
&&\wrho(x_1+w_1,\dots, x_{n-1}+w_{n-1})\wrho(y_1+w'_1,\dots, y_{n-1}+w'_{n-1})(m+v)\\
&&-\wrho(y_1+w'_1,\dots,y_{n-1}+w'_{n-1})\wrho(x_1+w_1,\dots, x_{n-1}+w_{n-1})(m+v)\\
&=&\wrho(x_1+w_1,\dots, x_{n-1}+w_{n-1})\Big( \rho(y_1, \dots, y_{n-1})(m)+\rho_{1}(y_1,\dots, y_{n-1})(v)\\
&&+\sumnoi (-1)^{n-i}\rho_{3}(y_1,\dots, \widehat{y_i},\dots,y_{n-1},m)(w'_i)\Big)-(x\leftrightarrow y, w_i\leftrightarrow w'_i)\\
&=&[\rho(x_1, \dots, x_{n-1}),\rho(y_1, \dots, y_{n-1})](m)+[\rho_1(x_1, \dots, x_{n-1}),\rho_{1}(y_1,\dots, y_{n-1})](v)\\
&&+\sumnoi (-1)^{n-i}\rho_1(x_1, \dots, x_{n-1})\rho_{3}(y_1,\dots, \widehat{y_i},\dots,y_{n-1},m)(w'_i)\\
&&+\sumnoi (-1)^{n-i}\rho_1(y_1, \dots, y_{n-1})\rho_{3}(x_1,\dots, \widehat{x_i},\dots,x_{n-1},m)(w_i)\\
&&+\sumnoj (-1)^{n-j}\sumnoi (-1)^{n-i}\rho_{3}(x_1,\dots, \widehat{x_j},\dots,x_{n-1},\rho(y_1, \dots, y_{n-1})(m))(w_j)\\
&&+\sumnoj (-1)^{n-j}\sumnoi (-1)^{n-i}\rho_{3}(y_1,\dots, \widehat{y_j},\dots,y_{n-1},\rho(x_1, \dots, x_{n-1})(m))(w'_j),
\end{eqnarray*}
and the right-hand side of \eqref{semi02} is
\begin{eqnarray*}
&&\sumnoi \wrho(y_1+w'_1,\dots, [x_1+w_1,\dots,x_{n-1}+w_{n-1},y_i+w'_i],\dots,y_{n-1}+w'_{n-1}))(m+v)\\
&=&\sumnoi \wrho\Big(y_1+w'_1,\dots, [x_1,\dots,x_{n-1},y_i]+\rho_2(x_1,\dots,x_{n-1})(w'_i)\\
&&+\sum_{j=1}^{n-1}(-1)^{n-j}\rho_2 (x_1,\dots, \widehat{x_j},\dots, y_i)(w_j),\dots,y_{n-1}+w'_{n-1})\Big)(m+v)\\
&=&\sumnoi\rho(y_1, \dots, [x_1,\dots,x_{n-1},y_i],\dots, y_{n-1})(m)\\
&&+\sumnoi\rho_{1}(y_1,\dots, [x_1,\dots,x_{n-1},y_i],\dots, y_{n-1})(v)\\
&&+\sumnoi\sum_{j=1}^{n-1} (-1)^{n-i}\rho_{3}(y_1,\dots, [x_1,\dots,x_{n-1},y_i],\dots, \widehat{y_j},\dots,y_{n-1},m)(w'_j)\\
&&+\sumnoi\sum_{j=1}^{n-1}(-1)^{n-j}\rho_{3}(y_1,\dots, \widehat{y_i},\dots,y_{n-1},m)\\
&&\quad\Big(\rho_2(x_1,\dots,x_{n-1})(w'_i)+\sum_{j=1}^{n-1,i}(-1)^{n-j}\rho_2 (x_1,\dots, \widehat{x_j},\dots, y_i)(w_j)\Big).
\end{eqnarray*}
Thus, the two sides of  \eqref{semi02} are equal to each other by $\rho, \rho_1 $ are representation of $\g$ on $M, V$ and \eqref{defrep04}, \eqref{defrep05} hold.

Finally, we  prove that $\widehat{\rho}$ is a representation of $\g\oplus W$ on $M\oplus V$ satisfying (R2):
\begin{eqnarray}\label{semi03}
\notag&&\wrho(x_1+w_1,\dots, x_{n-2}+w_{n-2}, [y_1+w'_1,\dots, y_{n}+w'_{n}])(m+v)\\
\notag&=&\sumn (-1)^{n-i}\wrho(y_1+w'_1,\dots, \widehat{y_i+w'_i},\dots,y_{n}+w'_{n})\\
&&\wrho(x_1+w_1,\dots, x_{n-2}+w_{n-2}, y_i+w'_i)(m+v).
\end{eqnarray}
The left-hand side of \eqref{semi03} is
\begin{eqnarray*}
\notag&&\wrho(x_1+w_1,\dots, x_{n-2}+w_{n-2}, [y_1+w'_1,\dots, y_{n}+w'_{n}])(m+v)\\
\notag&=&\wrho(x_1+w_1,\dots, x_{n-2}+w_{n-2},[y_1,\dots,y_{n}]\\
&&+ \sumn(-1)^{n-i} \rho (y_1,\dots, \widehat{y_i},\dots, y_n)(w'_i))(m+v)\\
\notag&=&\rho(x_1,\dots, x_{n-2},[y_1,\dots,y_{n}])(m)+\rho(x_1,\dots, x_{n-2},[y_1,\dots,y_{n}])(v)\\
\notag&&+\sumnoii (-1)^{n-i}\rho_{3}(x_1,\dots, \widehat{x_i},\dots,x_{n-2}, [y_1,\dots,y_{n}],m)(w_i)\\
\notag&&+\rho_{3}(x_1,\dots, x_{n-2},m)\left(\sumn(-1)^{n-i} \rho (y_1,\dots, \widehat{y_i},\dots, y_n)(w'_i)\right)
\end{eqnarray*}
and the right-hand side of \eqref{semi03} is
\begin{eqnarray*}
\notag&&\sumn (-1)^{n-i}\wrho(y_1+w'_1,\dots, \widehat{y_i+w'_i},\dots,y_{n}+w'_{n})\\
\notag&&\wrho(x_1+w_1,\dots, x_{n-2}+w_{n-2}, y_i+w'_i)(m+v)\\
\notag&=&\sumn (-1)^{n-i}\wrho(y_1+w'_1,\dots, \widehat{y_i+w'_i},\dots,y_{n}+w'_{n})\\
\notag&&\Big(\rho_1(x_1,\dots,x_{n-2},y_i)(m)+\rho_1(x_1,\dots,x_{n-2},y_i)(v)\\
\notag&&+\sumnojj(-1)^{n-j}\rho_3(x_1,\dots,\widehat{x_i}\dots, x_{n-2}, y_i, m)(w_j)+\rho_3(x_1,\dots, x_{n-2}, m)(w'_i)\Big)\\
\notag&=&\sumn(-1)^{n-i}\rho_1(y_1,\dots,\widehat{y_i}\dots, y_{n})\rho_1(x_1,\dots,x_{n-2},y_i)(m)\\
\notag&&+\sumn(-1)^{n-i}\rho_1(y_1,\dots,\widehat{y_i}\dots, y_{n})\rho_1(x_1,\dots,x_{n-2},y_i)(v)\\
\notag&&+\sumn(-1)^{n-i}\rho_1(y_1,\dots,\widehat{y_i}\dots, y_{n})\sumnojj(-1)^{n-j}\rho_3(x_1,\dots,\widehat{x_j}\dots, x_{n-2}, y_i, m)(w_j)\\
\notag&&+\sumn(-1)^{n-i}\rho_1(y_1,\dots,\widehat{y_i}\dots, y_{n})\rho_3(x_1,\dots, x_{n-2}, m)(w'_i)\\
\notag&&+\sumn(-1)^{n-i}\sumnojj(-1)^{n-j}\rho_3(y_1,\dots,\widehat{y_j},\dots,\widehat{y_i}\dots, y_{n}, \rho_1(x_1,\dots,x_{n-2},y_i)(m))(w'_j).
\end{eqnarray*}
Thus, the two sides of  \eqref{semi03} are equal to each other by $\rho, \rho_1 $ are representation of $\g$ on $M, V$ and \eqref{defrep06}, \eqref{defrep07} hold.
This completes the proof.
\end{proof}

\subsection{Cohomology}

Let $(M,\g, f)$ be  an $n$-Lie algebra in $\LM$ and $(V,W,\varphi)$ be a representation of  $(M,\g, f)$.
We define the $k$-cochain  $C^{k}((M,\g,f),(V,W,\varphi))$ to be the space
\begin{eqnarray*}
&&\Hom(\wedge^{k(n-1)+1} \g,W)\oplus\Hom(\wedge^{k(n-1)}  \g,\Hom(M,V ))\oplus \Hom(\wedge^{(k-1)(n-1)} \g,\Hom(M,W))
\end{eqnarray*}
which can be seen as the subspace of
\begin{eqnarray*}
&&\Hom(\otimes^{k}\L, \Hom(\g,W))\oplus\Hom(\otimes^{k}\L,\Hom(M,V ))\oplus \Hom(\otimes^{k-1}\L,\Hom(M,W)),
\end{eqnarray*}
and the coboundary map is given by
\begin{eqnarray}
D(\omega,\nu,\theta)=(-\delta_1\omega, -h_1\omega+\delta_2 \nu, -h_2\omega+\varphi_\sharp\nu-\delta_3\theta)
\end{eqnarray}
where $\delta_1,\delta_2,\delta_3$ are  coboundary maps in $n$-Lie algebra cohomology of $\g$ with coefficient in $W$, $\Hom(M,V)$, $\Hom(M,W)$ respectively.
The following maps:
\begin{eqnarray}
&&\varphi_\sharp: \Hom(\wedge^{k-1} \L,\Hom(M,V ))\to \Hom(\wedge^{k-1} \L,\Hom(M,W)),\\
&&h_1: \Hom(\wedge^k \L\otimes \g,W)\to \Hom(\wedge^{k+1} \L,\Hom(M,V)),\\
&&h_2: \Hom(\wedge^{k} \L\otimes \g,W)\to \Hom(\wedge^{k} \L,\Hom(M,W))
\end{eqnarray}
are defined by
\begin{eqnarray}
&&(\varphi_\sharp(\nu)(X_1, \dots, X_{k-1}))(m)=\varphi(\nu(X_1, \dots, X_{k-1})(m)),\\
&&(h_1(\omega)(X_1, \dots, X_{k}))(m)=\rho_3(m)\omega( X_1, \dots, X_{k}),\\
&&(h_2(\omega)(X_1, \dots, X_{k-1}))(m)=\omega(f(m), X_1, \dots, X_{k-1}).
\end{eqnarray}

It can be verified by direct computations that $D^2=0$, thus we obtain a cochain complex $C^k((M,\g,f),(V,W,\varphi))$ whose cohomology group
$$H^k((M,\g,f),(V,W,\varphi))=Z^k((M,\g,f),(V,W,\varphi))/B^k((M,\g,f),(V,W,\varphi))$$
is defined as cohomology group of $(M,\g,f)$ with coefficients in $(V,W,\varphi)$.

The cochain complex is precisely given  by
\begin{equation} \label{eq:complex}
\begin{split}
 & \qquad \qquad\  W\stackrel{D_0}{\longrightarrow}\\
 &  \quad\Hom(\g,W)\oplus\Hom(M,V)\stackrel{D_1}{\longrightarrow}\\
 & \Hom(\wedge^n \g, W)\oplus \Hom(\wedge^{n-1} \g,\Hom(M,V))\oplus \Hom(M,W)\stackrel{D_2}{\longrightarrow}\\
  & \Hom(\wedge^{2n-1} \g, W)\oplus \Hom(\wedge^{2n-2}\g,\Hom( M,V))\oplus \Hom(\wedge^{n-1}\g,\Hom(M,W))\stackrel{D_3}{\longrightarrow}\dots,
\end{split}
\end{equation}
where one writes the cochain complex in components as follows:
$$
\xymatrix{
  W \ar[d]_{-\delta_1} \ar[dr]_{ }& & & & \\
  \Hom(\g,W)\oplus \ar[d]_{-\delta_1}\ar[dr]_{-h_1} \ar[drr]^ {-h_2} &\hspace{-1cm} \Hom(M,V)  \ar[d]_{\delta_2}\ar[dr]^{\varphi_\sharp}&&& \\
\Hom(\wedge^n \g, W)\oplus\ar[d]_{-\delta_1}\ar[dr]_{-h_1} \ar[drr]^ {-h_2} & \Hom(\wedge^{n-1}\g,\Hom(M,V))\oplus
            \ar[d]_{\delta_2}  \ar[dr]^{\varphi_\sharp}    & \hspace{-.5cm}\Hom(M,W)\ar[d]_{-\delta_3}  &&\\
\Hom(\wedge^{2n-1} \g, W) \oplus\ar[d]_{-\delta_1}\ar[dr]_{-h_1}\ar[drr]^ {-h_2} & \Hom(\wedge^{2n-2}\g,\Hom( M,V)) \oplus
            \ar[d]_{\delta_2}  \ar[dr]^{\varphi_\sharp}    & \hspace{-.5cm} \Hom(\wedge^{n-1}\g,\Hom(M,W))\ar[d]_{-\delta_3} & &  \\
\Hom(\wedge^{3n-2} \g, W) \oplus    & \Hom(\wedge^{3n-3}\g,\Hom( M,V)) \oplus     &\Hom(\wedge^{2n-2}\g,\Hom(M,W)) & &  \\
        }
$$

$\vdots$

For 1-cochain $(N_0,N_1)\in \Hom(\g,W)\oplus\Hom(M,V)$, the coboundary map is
\begin{eqnarray*}
D_1(N_0,N_1)(m)&=&(-h_2N_0+\varphi_\sharp  N_1)(m)=\varphi\circ N_1(m)-N_0\circ f(m),\\[1em]
D_1(N_0,N_1)(x_1, \dots, x_n)&=&-\delta_1N_0(x_1, \dots, x_n)\notag\\
&=&[N_0(x_1), \dots, x_n] +\dots+[x_1, \dots, N_0x_n] - N_0[x_1, \dots, x_n],\\[1em]
D_1(N_0,N_1)(x_1, \dots, x_{n-1}, m)&=&(-h_1N_0+\delta_2 N_1)(x_1, x_2, m)\notag\\
&=&[N_0(x_1), \dots, x_{n-1}, m] +\dots + [x_1, \dots, N_0x_{n-1},  m] \\
&&+ [x_1, \dots, x_{n-1},  N_1m]- N_1[x_1, \dots, x_{n-1}, m].
\end{eqnarray*}
Thus, a 1-cocycle is   $(N_0,N_1)\in \Hom(\g,W)\oplus\Hom(M,V)$, such that
\begin{eqnarray}
&&\varphi\circ N_1=N_0\circ f,\\
&&N_0[x_1, \dots, x_n]=\sumn [x_1, \dots, N_0x_i, \dots, x_n],\\
\notag&&N_1[x_1, \dots, x_{n-1}, m]=\sumnoi [x_1, \dots, N_0x_i, \dots, x_{n-1}, m]+[x_1, \dots, x_{n-1}, N_1m].
\end{eqnarray}

For a 2-cochain $(\omega,\nu,\theta)\in \Hom(\wedge^n \g, W)\oplus \Hom(\g,\Hom(M,V))\oplus \Hom(M,W)$, we get that
$$D_2(\omega,\nu,\theta)\in  \Hom(\wedge^{2n-1} \g, W)\oplus \Hom(\wedge^{2n-2}\g,\Hom( M,V))\oplus \Hom(\wedge^{n-1}\g,\Hom(M,W))$$
where
\begin{eqnarray*}
&&D_2(\omega,\nu,\theta)(x_1, \dots, x_{n-1}, y_1, \dots, y_n)\\
&=&-\delta_1\omega(x_1, \dots, x_{n-1}, y_1, \dots, y_n),
\end{eqnarray*}
\begin{eqnarray*}
&&D_2(\omega,\nu,\theta)(x_1, \dots, x_{n-1}, y_1, \dots, y_{n-1}, m)\\
&=&(-h_1\omega+\delta_2 \nu)(x_1, \dots, x_{n-1}, y_1, \dots, y_{n-1}, m),
\end{eqnarray*}
\begin{eqnarray*}
&&D_2(\omega,\nu,\theta)(x_1, \dots, x_{n-1}, m)\\
&=&(-h_2\omega+\varphi_\sharp\nu-\delta_3\theta)(x_1,\dots, x_{n-1}, m).
\end{eqnarray*}
Thus, a 2-cocycle is  $(\omega,\nu,\theta)\in \Hom(\wedge^n \g, W)\oplus \Hom(\wedge^{n-1} \g,\Hom(M,V))\oplus \Hom(M,W)$ such that the following conditions holds:
\begin{eqnarray}
\nonumber &&\theta (\rho(x_1,\dots, x_{n-1})(m))+\varphi (\nu (x_1, \dots, x_{n-1}, m))\\
&=&\omega (x_1, \dots, x_{n-1},  f(m))+[x_1, \dots, x_{n-1}, \theta (m)],
\label{eq:2-coc03}
\end{eqnarray}
\begin{eqnarray}
\nonumber &&\omega(x_1, \dots, x_{n-1}, [y_1, \dots, y_n])+\rho_2(x_1, \dots, x_{n-1})\omega(y_1, \dots, y_n)\\
\nonumber &=&\sumn \omega(y_1,\dots,[ x_1, \dots, x_{n-1}, y_i],\dots, y_n)\\
\label{eq:2-coc01} &&+\sumn(-1)^{n-i}\rho_2(y_1,\dots,\widehat{y_i},\dots y_n)\omega(x_1, \dots, x_{n-1}, y_i),
\end{eqnarray}
\begin{eqnarray}
\nonumber &&\nu(x_1, \dots, x_{n-1}, [y_1, \dots, y_{n-1},m])+\rho_1(x_1, \dots, x_{n-1})\nu(y_1, \dots, y_{n-1}, m)\\
\nonumber &=&\sumnoi\nu(y_1,\dots, [ x_1, \dots, x_{n-1}, y_i], \dots, y_{n-1}, m) +\nu(y_1, \dots, y_{n-1}, [ x_1, \dots, x_{n-1}, m])\\
\nonumber &&+\sumnoi [y_1,\nu(x_1, \dots, x_{n-1},y_i), \dots, y_{n-1}, m]+\rho_1(y_1,\dots, y_{n-1})\nu(x_1,  \dots, x_{n-1},  m).\\
\label{eq:2-coc02}
\end{eqnarray}

\section{Abelian extensions}

\begin{defi}
Let $(M,\g, f)$ be an $n$-Lie algebra in $\LM$. An extension of $(M,\g, f)$ is a
short exact sequence
\begin{equation}\label{eq:ext1}
\CD
  0 @>0>>  \frkh @>i_1>> \widehat{M} @>p_1>>  M  @>0>> 0 \\
  @V 0 VV @V \varphi VV @V \widehat{\dM} VV @V\dM VV @V0VV  \\
  0 @>0>> W @>i_0>> \widehat{\g} @>p_0>> \g @>0>>0
\endCD
\end{equation}
where $(V,W,\varphi)$ is an $n$-Lie algebra in $\LM$.

We call $(\widehat M,\widehat\g, \widehat f)$ an extension of $(M,\g, f)$ by
$(V,W,\varphi)$, and denote it by $\widehat{\E}$.
It is called an abelian extension if $(V,W,\varphi)$ is an abelian $n$-Lie algebra in $\LM$, which means that $\varphi: V\to W$ is a linear map between vector spaces $V$ and $W$ with zero bracket and trivial representation of $V$ on $W$.
\end{defi}

A splitting $\sigma=(\sigma_0,\sigma_1):(M,\g, f)\to (\widehat M,\widehat \g, \widehat f)$ consists of linear maps
$\sigma_0:\mathfrak{g}\to\widehat{\g}$ and $\sigma_1:M\to \widehat M$
 such that  $p_0\circ\sigma_0=\id_{\mathfrak{g}}$ and  $p_1\circ\sigma_1=\id_{M}$.

 Two extensions
 $$\widehat{\E}:0\longrightarrow(V,W,\varphi)\stackrel{i}{\longrightarrow}(\widehat M,\widehat \g, \widehat f)\stackrel{p}{\longrightarrow}(M,\g, f)\longrightarrow0$$
 and
 $$\widetilde{\E}:0\longrightarrow(V,W,\varphi)\stackrel{j}{\longrightarrow}(\widetilde{M},\widetilde{\g}, \widetilde{f})\stackrel{q}{\longrightarrow}(M,\g, f)\longrightarrow0$$
 are equivalent, if there exists a homomorphism $F:(\widehat M,\widehat \g, \widehat f)\longrightarrow(\widetilde{M},\widetilde{\g}, \widetilde{f})$ such that $F\circ i=j$, $q\circ
 F=p$.

Let $(\widehat M,\widehat\g, \widehat f)$ be an abelian extension of $(M,\g, f)$ by
$(V,W,\varphi)$ and $\sigma:(M,\g, f)\to (\widehat M,\widehat \g, \widehat f)$ be a splitting.
Define a  representation of $(M,\g, f)$ over $(V,W,\varphi)$ by
\begin{eqnarray*}
&&
\left\{\begin{array}{ll}
&\rho_{1}: \wedge^{n-1}\g\longrightarrow \gl(V),\\
& \rho_{1}(x_1, \dots, x_{n-1})(v)\triangleq\widehat\rho(\sigma_0(x_1), \dots, \sigma_0(x_{n-1})(v),
\end{array}\right.\\
&&
\left\{\begin{array}{ll}
&\rho_{2}: \wedge^{n-1}\g\longrightarrow \gl(W),\\
& \rho_{2}(x_1, \dots, x_{n-1})(z)\triangleq[\sigma_0(x_1), \dots, \sigma_0(x_{n-1}), w]_{\hg},
\end{array}\right.\\
&&
\left\{\begin{array}{ll}
&\rho_{3}:  \g\otimes\dots\otimes \g\otimes M\longrightarrow \Hom(W,V),\\
&\rho_3(x_1,\dots,x_{n-2},m)(w)\triangleq-\widehat\rho(\sigma_0(x_1), \dots, \sigma_0(x_{n-2}),  w)(\sigma_1(m)).
\end{array}\right.
\end{eqnarray*}
for all $x_1, \dots, x_{n-1},z\in\frkg$, $m\in M $.

\begin{pro}\label{pro:2-modules}
With the above notations, $(\rho_1,\rho_2,\rho_3)$ is a representation of $(M,\g,f)$ on $(V,W,\varphi)$.
Furthermore,  $(\rho_1,\rho_2,\rho_3)$ does not depend on the choice of the splitting $\sigma$.
Moreover,  equivalent abelian extensions give the same representation.
\end{pro}

\begin{proof}
Firstly, we show that $\rho_{1},\rho_{2},\rho_{3}$ are well defined.
Since $\Ker p_{0} \cong W$, then for $w\in {W}$, we have $p_{0}(w)=0$.
By the fact that $(p_1,p_0)$ is a homomorphism between $(\widehat M,\widehat\g, \widehat f)$ and $(M,\g,f)$, we get
\begin{eqnarray*}
p_{0}[\sigma_0(x_1),\dots,\sigma_0(x_{n-1}),w]_{\wg}&=&[p_{0}\sigma_0(x_1),\dots,p_{0}\sigma_0(x_{n-1}),p_{0}(w)]_{\wg}\\
&=&[p_{0}\sigma_0(x_1),\dots,p_{0}\sigma_0(x_{n-1}), 0]_{\wg}=0.
\end{eqnarray*}
Thus, $\rho_{2}(x_1, \dots, x_{n-1})(w)=[\sigma_0(x_1),\dots,\sigma_0(x_{n-1}),w]_{\wg}\in  \ker p_{0} \cong {W}$.
Similar computations show that $\rho_{1}(x_1, \dots, x_{n-1})(v)=\wrho(\sigma_0(x_1), \dots, \sigma_0(x_{n-1}))(v), \rho_{3}(x_1,\dots,x_{n-2},m)(w)=-\wrho(\sigma_0(x_1), \dots, \sigma_0(x_{n-2}),w)(\sigma_1(m))\in \Ker p_0=V$.

Now we  will show that $\rho_{i}$ are independent of the choice of $\sigma$.

In fact, if we choose another section $\sigma'_0:\g\to\hg$, then
\begin{eqnarray*}
&&p_0(\sigma_0(x_i)-\sigma'_0(x_i))=x_i-x_i=0\\
&\Longrightarrow&\sigma_0(x_i)-\sigma'_0(x_i)\in \Ker p_0=W\\
&\Longrightarrow&\sigma'_0(x_i)=\sigma'_0(x_i)+w_i
\end{eqnarray*}
for some $w_i\in W$.

Since we  have $[\dots,w,w']_{\hat{\g}}=0$ for all $w,w'\in W$, which implies that
\begin{eqnarray*}
&&[\sigma_0'(x_1),\dots,\sigma_0'(x_{n-1}),w]_{\hat{\g}}\\
&=&[\sigma_0(x_1)+u_1,\dots,\sigma_0(x_{n-1})+u_{n-1},w]_{\hat{\g}}\\
&=&[\sigma_0(x_1),\dots,\sigma_0(x_{n-1})+u_{n-1},w]_{\hat{\g}}+[u_1,\dots,\sigma_0(x_{n-1})+u_{n-1},w]_{\hat{\g}}\\
&=&[\sigma_0(x_1),\dots,\sigma_0(x_{n-1}),w]_{\hat{\g}}+\dots+[\sigma_0(x_1),\dots,u_{n-1},w]_{\hat{\g}}\\
&=&[\sigma_0(x_1),\dots,\sigma_0(x_{n-1}),w]_{\hat{\g}},
\end{eqnarray*}
thus, $\rho_2$ is independent on the choice of $\sigma$.
If we choose another splitting
$\sigma'_1: M\to \widehat M$, then $p_1(\sigma_1(m)-\sigma'_1(m))=m-m=0$, i.e. $\sigma_1(m)-\sigma'_1(m)\in \Ker p_1=V$.
Similar argument shows that $\rho_1, \rho_3$ are independent on the choice of $\sigma$.
Thus $\rho_{1},\rho_{2},\rho_{3}$ are well defined.

Secondly, we check that $\rho=(\rho_{1},\rho_{2},\rho_{3})$ is indeed a representation of $(M,\g,f)$ on $(V,W,\varphi)$.
Since $(V,W,\varphi)$ is an abelian $n$-Lie algebra in $\LM$, we have
\begin{eqnarray*}
&&\rho_2(x_1,\dots,x_{n-1})\rho_2(y_1,\dots,y_{n-1})(w)\\
&=&[\sigma_0(x_1), \dots, \sigma_0(x_{n-1}),[\sigma_0(y_1), \dots, \sigma_0(y_{n-1}), w]]_{\hg}\\
&=&[\sigma_0(y_1), \dots, [\sigma_0(x_1), \dots, \sigma_0(x_{n-1}),\sigma_0(y_i)]\dots, \sigma_0(y_{n-1}), w]]_{\hg}\\
&&+[\sigma_0(y_1), \dots, \sigma_0(y_{n-1}),[\sigma_0(x_1), \dots, \sigma_0(x_{n-1}), w]]_{\hg}\\
&=&\sumnoi \rho_2(y_1,\dots, [x_1,\dots,x_{n-1},y_i],\dots,y_{n-1}))(w)\\
&&+\rho_2(y_1,\dots,y_{n-1})\rho_2(x_1,\dots,x_{n-1})(w),
 \end{eqnarray*}
which implies that $\rho_2$ is a representation of  $\g$ on $W$.
Similarly, we get $\rho_1$ is a representation of  $\g$ on $V$.
For the equivariant between $\rho_{1}$ and $\rho_{2}$, we have
\begin{eqnarray*}
\varphi\circ\rho_{1}(x_1,\dots,x_{n-1})(v)&=&\varphi\wrho(\sigma_0(x_1),\dots,\sigma_0(x_{n-1}))(v)\\
&=&[\sigma_0(x_1),\dots,\sigma_0(x_{n-1}),\varphi(v)]\\
&=&\rho_{2}(x_1,\dots,x_{n-1})\circ\varphi(v).
\end{eqnarray*}
For $\rho_{3}: \g\otimes\dots\otimes\g\otimes M\to \Hom(W,V)$, we have
\begin{eqnarray*}
\varphi\circ\rho_{3}(x_1,\dots, x_{n-2}, m)(w)
&=&\varphi[\sigma_0(x_1), \dots, \sigma_0(x_{n-2}), w,\sigma_1(m)]\\
&=&-[\sigma_0(x_1), \dots, \sigma_0(x_{n-2}), w,\widehat{f}\sigma_1(m)]\notag\\
&=&-[\sigma_0(x_1), \dots, \sigma_0(x_{n-2}),w,\sigma_0(f(m))]\\
&=&[\sigma_0(x_1), \dots, \sigma_0(x_{n-2}),\sigma_0(f(m)),w]\\
&=&\rho_{2}(x_1,\dots, x_{n-2}, f(m))(w).
\end{eqnarray*}
Thus, $(\rho_1,\rho_2,\rho_3) $ is a representation of  $(M,\g,f)$ over $(V,W,\varphi)$.

Finally, suppose that $\widehat{\E}$ and $\widetilde{\E}$ are equivalent abelian extensions,
and $F:(\widehat M,\widehat \g, \widehat f)\longrightarrow(\widetilde{M},\widetilde{\g}, \widetilde{f})$ be the morphism.
Choose linear sections $\sigma$ and $\sigma'$ of $p$ and $q$. Then we have $q_0F_0\sigma_0(x_i)=p_0\sigma_0(x_i)=x_i=q_0\sigma'_0(x_i)$,
thus $F_0\sigma_0(x_i)-\sigma'_0(x_i)\in \Ker q_0=W$. Therefore, we obtain
\begin{eqnarray*}
[\sigma'_0(x_1),\dots,\sigma'_0(x_{n-1}),v+w]_{\tg}&=&[F_0\sigma_0(x),\dots,F_0\sigma'_0(x_{n-1}),v+w]_{\tg}\\
&=&F_0[\sigma_0(x_1),\dots,\sigma_0(x_{n-1}),v+w]_{\hg}\\
&=&[\sigma_0(x),\dots,\sigma_0(x_{n-1}),v+w]_{\hg},
\end{eqnarray*}
which implies that equivalent abelian extensions give the same $\rho_{1}$ and $\rho_{2}$. Similarly, we can show that equivalent abelian extensions also give the same $\rho_{3}$.
Therefore, equivalent abelian extensions also give the same representation. The proof is finished.
\end{proof}

Let $\sigma:(M,\g, f)\to (\widehat M,\widehat \g, \widehat f)$  be a section of the abelian extension. Define the following linear maps:
$$
\left\{
\begin{array}{ll}
&\theta: M \longrightarrow W,\quad  \theta(m)\triangleq\widehat{\dM}\sigma_1(m)-\sigma_0f(m),\\
&\omega:\wedge^n\frkg\longrightarrow W,\quad \omega(x_1, \dots, x_n)\triangleq[\sigma_0(x_1), \dots, \sigma_0(x_n)]_{\widehat{\g}}-\sigma_0([x_1, \dots, x_n]_{\g}),\\
&\nu:\wedge^{n-1}\frkg\otimes M\to V,\\
& \nu(x_1, \dots, x_{n-1}, m)\triangleq \widehat{\rho}(\sigma_0(x_1),\dots, \sigma_0(x_{n-1}))(\sigma_1(m))-\sigma_1\rho(x_1, \dots, x_{n-1})(m).
\end{array}\right.
$$
for all $x_i\in\frkg$, $m\in M $.

\begin{thm}\label{thm:2-cocylce}
With the above notations, $(\theta,\omega,\nu)$ is a $2$-cocycle of $(M,\g, f)$ with coefficients in $(V,W,\varphi)$.
\end{thm}

\begin{proof}
By the equality
$$\widehat{\dM}(\widehat{\rho}(\sigma_0 x_1,\dots, \sigma_0x_{n-1})(\sigma_1 (m)))=[\sigma_0 x_1,\sigma_0 x_{n-1},\widehat{\dM}\sigma_1 (m)]_{\widehat{\mathfrak{g}}},$$
we obtain that
\begin{eqnarray}
\nonumber &&\omega (x_1, \dots, x_{n-1},  f(m))+\rho_1(x_1, \dots, x_{n-1}, \theta (m))\\
&=&\theta ([x_1,\dots, x_{n-1}, m])+\varphi (\nu (x_1, \dots, x_{n-1}, m)).
\label{eq:2-coc03}
\end{eqnarray}

By the equality
\begin{eqnarray*}
[\si_0 x_1,\dots,\si_0 x_{n-1}, [\si_0 y_1,\dots,\si_0 y_n]]
&=&\sumn [\si_0 y_1,\dots,[\si_0 x_1,\dots,\si_0 x_{n-1},\si_0 y_i],\dots,\si_0 y_n],
\end{eqnarray*}
we get that the left-hand side is equal to
\begin{eqnarray*}
&=&[\si_0 x_1,\dots,\si_0 x_{n-1}, \omega(y_1,\dots,y_n)+\si_0([y_1,\dots,y_n])] \\
&=&\rho(x_1,\dots,x_{n-1})\omega(y_1,\dots,y_n)+[\si_0 x_1,\dots,\si_0 x_{n-1},\si_0([y_1,\dots,y_n])]\\
&=&\rho(x_1,\dots,x_{n-1})\omega(y_1,\dots,y_n)+\omega(x_1,\dots,x_{n-1},[y_1,\dots,y_n])\\
&&+\si_0([x_1,\dots,x_{n-1},[y_1,\dots,y_n]]),
\end{eqnarray*}
the right-hand side is equal to
\begin{eqnarray*}
&=&\sumn(-1)^{n-i}[\si_0 y_1,\dots,\si_0 \widehat{y_i} ,\dots,\si_0 y_n, [\si_0 x_1,\dots,\si_0 x_{n-1},\si_0 y_i]]\\
&=&\sumn(-1)^{n-i}[\si_0 y_1,\dots,\si_0 \widehat{y_i} ,\dots,\si_0 y_n,,\omega(x_1,\dots,x_{n-1},y_i)+\si_0([x_1,\dots, x_{n-1},y_i])] \\
&=&\sumn(-1)^{n-i}\rho(y_1,\dots,\widehat{y_i},\dots,y_n)\omega(x_1,\dots,x_{n-1},y_i)\\
&&+\omega(\sumn[y_1,\dots,[x_1,\dots,x_{n-1},y_i],\dots,y_n]])\\
&&+\si_0(\sumn[y_1,\dots,[x_1,\dots,x_{n-1},y_i],\dots,y_n]]).
\end{eqnarray*}
Thus, we obtain
\begin{eqnarray}
\nonumber&& \omega(x_1,\dots,x_{n-1},[y_1,\dots,y_{n}])+\rho(x_1,\dots,x_{n-1})\omega(y_1,\dots,y_{n})\\
\nonumber&=&\sumn\omega([y_1,\dots,[x_1,\dots,x_{n-1},y_i],\dots,y_n]])\\
&&+\sumn(-1)^{n-i}\rho(y_1,\dots,\widehat{y_i},\dots,y_n)\omega(x_1,\dots,x_{n-1},y_i).
\end{eqnarray}

Finally,  by the equality
\begin{eqnarray*}
&&\widehat\rho(\si_0 x_1,\dots, \si_0x_{n-1})\rho(\si_0y_1,\dots,\si_0y_{n-1})(\si_1 m)\\
&=&\sumnoi \rho(\si_0 y_1,\dots, [\si_0 x_1,\dots, \si_0 x_{n-1}, \si_0y_i],\dots, \si_0y_{n-1}))(\si_1 m)\\
&&+\rho(\si_0 y_1,\dots,\si_0 y_{n-1})\rho(\si_0 x_1,\dots, \si_0 x_{n-1})(\si_1 m),
\end{eqnarray*}
we obtain that
\begin{eqnarray}
\nonumber &&\nu(x_1, \dots, x_{n-1}, [y_1, \dots, y_{n-1},m])+[x_1, \dots, x_{n-1},\nu(y_1, \dots, y_{n-1}, m)]\\
\nonumber &=&\sumnoi\nu(y_1,\dots, [ x_1, \dots, x_{n-1}, y_i], \dots, y_{n-1}, m) +\nu(y_1, \dots, y_{n-1}, [ x_1, \dots, x_{n-1}, m])\\
\nonumber&&+\sumnoi [y_1,\nu(x_1, \dots, x_{n-1},y_i), \dots, y_{n-1}, m]+[y_1,\dots, y_{n-1}, \nu(x_1,  \dots, x_{n-1},  m)].\\
\end{eqnarray}
Thus $(\theta,\omega,\nu)$ is a $2$-cocycle.
\end{proof}

Now, we define the $n$-Lie algebra structure in $\LM$
 using the 2-cocycle given above. More precisely, we have
\begin{eqnarray*}
&&
\left\{\begin{array}{rl}
&\widehat{f}(m+v)\triangleq\dM(m)+\theta(v)+\varphi(v),\\
\end{array}\right.\\
&&
\left\{\begin{array}{rl}
&{[x_1+w_1, \dots, x_n+w_n]}\\
&\triangleq[x_1, \dots, x_n]+\omega(x_1, \dots, x_n)+\sumn(-1)^{n-i} \rho_2 (x_1,\dots, \widehat{x_i},\dots, x_n)(w_i),\\
\end{array}\right.\\
&&
\left\{\begin{array}{rl}
&{\wrho(x_1+w_1,\dots, x_{n-1}+w_{n-1})(m+v)}\\
&\triangleq\rho(x_1, \dots, x_{n-1})(m)+\nu(x_1, \dots, x_{n-1}, m)\\
&\quad+\rho_{1}(x_1,\dots, x_{n-1})(v)
+\sumnoi (-1)^{n-i}\rho_{3}(x_1,\dots, \widehat{x_i},\dots,x_{n-1},f(m))(w_i),
\end{array}\right.
\end{eqnarray*}
for all $x_i\in\mathfrak{g}$, $m\in M$,
$v\in V$ and $w_i\in W$. Thus any extension
$E_{\widehat{\g}}$ given by \eqref{eq:ext1} is isomorphic to
\begin{equation}\label{eq:ext2}
\CD
  0 @>0>>  V @>i_1>>  M \oplus V@>p_1>>  M  @>0>> 0 \\
  @V 0 VV @V \varphi VV @V \widehat{\dM} VV @V\dM VV @V0VV  \\
  0 @>0>>  W  @>i_0>> \g \oplus W @>p_0>> \g @>0>>0,
\endCD
\end{equation}
where the $n$-Lie algebra structure in $\LM$  is given as above.

\begin{thm}
There is a one-to-one correspondence between equivalence classes of abelian extensions and  elements of  the second cohomology group $\mathbf{H}^2((M,\g,f),(V,W,\varphi))$.
\end{thm}

\pf
Let $\widehat{\E}'$ be another abelian extension determined by the 2-cocycle $(\theta',\omega',\nu')$. We are going to show that $\widehat{\E}$ and $\widehat{\E}'$ are equivalent if and only if 2-cocycles  $(\theta,\omega,\nu)$ and $(\theta',\omega',\nu')$ are in the same cohomology class.

 Since $F=(F_0,F_1)$ is an equivalence of extensions, there exist two linear maps $b_0:\g \longrightarrow W$
and $b_1: M \longrightarrow V$ such that
 $$F_0(x_i+w)=x_i+b_0(x_i)+w,\quad F_1(m+v)=m+b_1(m)+v.$$

First, by the equality
\begin{eqnarray*}
\label{eqn:fi} \widehat{f}'F_1(m)&=&F_0\widehat{f}(m),
\end{eqnarray*}
we have
\begin{equation}\label{eq:exact1}
\theta(m)-\theta'(m)=\varphi b_1(m)-b_0(\dM (m)).
\end{equation}

Second, by direct computations we have
\begin{eqnarray*}
&&F[x_1,\dots,x_{n}]_{\omega}\\
&=&F_1([x_1,\dots,x_{n}]+\omega(x_1,\dots,x_{n}))\\
&=&[x_1,\dots,x_{n}]+\omega(x_1,\dots,x_{n})+b_0([x_1,\dots,x_{n}])
\end{eqnarray*}
and
\begin{eqnarray*}
&&[F(x_1),\dots,F(x_n)]_{\omega'}\\
&=&[x_1+b_0(x_1),\dots,x_n+b_0(x_n)]_{\omega'}\\
&=&[x_1,\dots,x_n]+\omega'(x_1,\dots,x_n)+\sumn(-1)^{n-i}\rho(x_1,\dots,\widehat{x_i}\dots,x_n)b_0(x_i).
\end{eqnarray*}
Thus, we have
\begin{eqnarray}
\notag&&(\omega-\omega')(x_1,\dots,x_{n})\\
\label{eq:exact2}&=&\sumn(-1)^{n-i}\rho(x_1,\dots,\widehat{x_i}\dots,x_n)b_0(x_i)-b_0([x_1,\dots,x_{n}]).
\end{eqnarray}

Similarly, by the equality
\begin{eqnarray*}
&&F_1([x_1,\dots, x_{n-1}, m]+\nu(x_1, \dots, x_{n-1}, m))\\
&=&[x_1,\dots, x_{n-1}, m]+\nu(x_1, \dots, x_{n-1}, m)-b_1(x_1, \dots, x_{n-1}, m)
\end{eqnarray*}
and
\begin{eqnarray*}
&&\rho'(F_0(x_1),\dots,F_0(x_{n-1}))(F_1(m))\\
&=&[x_1,\dots, x_{n-1}, m]+\nu'(x_1, \dots, x_{n-1},  m)\\
&&-\sumnoi [x_1, \dots, b_0x_i, \dots, x_{n-1}, m]-\rho_1(x_1, \dots, x_{n-1})(b_1m).
\end{eqnarray*}
Thus  we get
\begin{eqnarray}
\notag&&(\nu-\nu')(x_1, \dots, x_{n-1},  m)\\
\notag&=&b_1\rho(x_1, \dots, x_{n-1})(m)-\sumnoi [x_1, \dots, b_0x_i, \dots, x_{n-1}, m]-\rho_1(x_1, \dots, x_{n-1})(b_1m).\\
\label{eq:exact3}
\end{eqnarray}

By \eqref{eq:exact1}-\eqref{eq:exact3}, we deduce that $(\psi,\omega,\nu)-(\psi',\omega',\nu')=D(b_0,b_1)$. Thus, they are in the same cohomology class.

Conversely, if  $(\theta,\omega,\nu)$ and $ (\theta',\omega',\nu')$ are in the same cohomology class,
assume that $(\psi,\omega,\nu)-(\psi',\omega',\nu')=D(b_0,b_1)$.
Then define $(F_0,F_1)$ by
 $$F_0(x+w)=x+b_0(x)+w,\quad F_1(m+v)=m+b_1(m)+v.$$
Similar as the above proof, we can show that $(F_0,F_1)$ is an equivalence. We omit the details.
\qed

\section{Infinitesimal deformations}
In this section, we explore infinitesimal deformations of $n$-Lie algebras in $\LM$ specifically focusing on the case of $n=3$. The results presented in this section are applicable to the general $n$ case. For more information on the theory of infinitesimal deformations for both Lie algebras and $3$-Lie algebras, one refers to \cite{D,FF,KM}.

Let $(M,\g, f)$ be a $3$-Lie algebra in $\LM$ and $\theta: M \to\g,~\omega:\wedge^3\g\to \g ,~\nu:\wedge^2\g \otimes M \to M$ be linear maps. Consider a $\lambda$-parametrized family of linear operations:
 \begin{eqnarray*}
\dM\dlam (m)&\triangleq&\dM (m)+\lambda\theta (m),\\
{[x_1, x_2, x_3]}\dlam&\triangleq& [x_1, x_2, x_3]+\lambda\omega(x_1, x_2, x_3),\\
{[x_1, x_2, m]}\dlam&\triangleq& [x_1, x_2, m]+ \lambda\nu(x_1, x_2, m),
 \end{eqnarray*}
 where we denote $[x_1, x_2, m]\triangleq\rho(x_1, x_2)(m)$ in the following of this section for simplicity.

If $(M_\lambda, \g\dlam, f\dlam)$ forms a 3-Lie algebra in $\LM$, then we say that
$(\theta,\omega ,\nu)$ generates a 1-parameter infinitesimal deformation of $(M,\g, f)$.

\begin{thm}\label{thm:deformation} With the notations above,
$(\theta,\omega ,\nu)$ generates a $1$-parameter infinitesimal deformation of $(M,\g, f)$  is equivalent to the following conditions:
\begin{itemize}
  \item[\rm(i)] $(\theta ,\omega ,\nu)$ is a $2$-cocycle of $\frkg$ with coefficients in the adjoint representation;

  \item[\rm(ii)] $(M,\g,\theta)$ is a $3$-Lie algebra in $\LM$ with brackets $\omega$ and $\nu$.
\end{itemize}
\end{thm}

\begin{proof}
If $(M_\lambda, \g\dlam, f\dlam)$  is a $3$-Lie algebra in $\LM$, then by Definition \ref{defi:Lie 2}, $f\dlam$ is an equivariant map, thus we have
\begin{eqnarray*}
&&\dM\dlam ([x_1, x_2,m]\dlam)-[x_1, x_2, \dM\dlam (m)]\dlam\\
&=&(\dM+\lambda\theta )([x_1, x_2,m]+ \lambda\nu(x_1, x_2, m))\\
&&-[x_1, x_2, \dM (m)+\lambda\theta (m)]-\lambda\omega (x_1, x_2, \dM (m)+\lambda\theta (m))\\
&=&\dM([x_1, x_2,m])+\lambda(\theta [x_1, x_2,m]+\dM\nu(x_1, x_2,m))+\lambda^2\theta \omega (x_1, x_2,m)\\
&&-[x_1, x_2,\dM (m)]-\lambda(\omega (x_1, x_2, f(m))+[x_1, x_2,\theta (m)])-\lambda^2\nu(x_1, x_2,\theta (m))\\
&=&0,
\end{eqnarray*}
which implies that
\begin{eqnarray}
\label{eq:2-coc01}\theta ([x_1, x_2,m])+\dM \nu (x_1, x_2,m)-\omega (x_1, x_2, \dM (m))-[x_1, x_2,\theta (m)]&=&0,\\
\label{eq:2-coc02}\theta \nu(x_1, x_2,m)-\omega (x_1, x_2,\theta (m))&=&0.
\end{eqnarray}

Since $\g\dlam$ is a 3-Lie algebra, we have the equality
\begin{eqnarray*}
&&[x_1, x_2, [y_1, y_2, y_3]\dlam]\dlam \\
&=&[[ x_1, x_2, y_1]\dlam , y_2, y_3]\dlam  + [y_1, [ x_1, x_2, y_2]\dlam , y_3]\dlam  + [y_1, y_2, [ x_1, x_2, y_3]\dlam]\dlam,
\end{eqnarray*}
the left-hand side is equal to
\begin{eqnarray*}
&&[x_1, x_2, [y_1, y_2, y_3]+\lam\omega(y_1, y_2, y_3)]\dlam\\
&=&[x_1, x_2, [y_1, y_2, y_3]]+\lam\omega(x_1, x_2, [y_1, y_2, y_3])\\
&&+[x_1, x_2,\lam\omega(y_1, y_2, y_3)]+\lam\omega(x_1, x_2,\lam\omega(y_1, y_2, y_3))\\
&=&[x_1, x_2, [y_1, y_2, y_3]]+\lam\{\omega(x_1, x_2, [y_1, y_2, y_3])+[x_1, x_2,\omega(y_1, y_2, y_3)]\}\\
&&+\lam^2\omega(x_1, x_2,\omega(y_1, y_2, y_3)),
\end{eqnarray*}
and the right-hand side is equal to
\begin{eqnarray*}
&&[[x_1, x_2, y_1]+\lam\omega(x_1, x_2, y_1), y_2, y_3]\dlam  + [y_1, [ x_1, x_2, y_2]+\lam\omega(x_1, x_2, y_2), y_3]\dlam \\
&&+ [y_1, y_2, [ x_1, x_2, y_3]+\lam\omega(x_1, x_2, y_3)]\dlam\\
&=&[[ x_1, x_2, y_1] , y_2, y_3]  + [y_1, [ x_1, x_2, y_2] , y_3] + [y_1, y_2, [ x_1, x_2, y_3]]\\
&&+\lam\{\omega([x_1, x_2, y_1],y_2, y_3) +[\omega(x_1, x_2, y_1), y_2, y_3]\\
&&\qquad+\omega(y_1,[ x_1, x_2, y_2], y_3)+[y_1,\omega(x_1, x_2, y_2), y_3]\\
&&\qquad+\omega(y_1, y_2, [ x_1, x_2, y_3])+[y_1,y_2,\omega(x_1, x_2, y_3)] \}\\
&&+\lam^2\{\omega(\omega( x_1, x_2, y_1), y_2, y_3) + \omega(y_1, \omega( x_1, x_2, y_2), y_3) + \omega(y_1, y_2, \omega( x_1, x_2, y_3))\}.
\end{eqnarray*}
Thus, we have
\begin{eqnarray}
\nonumber &&\omega(x_1, x_2, [y_1, y_2, y_3])+[x_1, x_2,\omega(y_1, y_2, y_3)]\\
\nonumber &=&\omega([x_1, x_2, y_1],y_2, y_3)+\omega(y_1,[ x_1, x_2, y_2], y_3) +\omega(y_1, y_2, [ x_1, x_2, y_3])\\
\label{eq:2-coc03} &&+[\omega(x_1, x_2, y_1), y_2, y_3]+[y_1,\omega(x_1, x_2, y_2), y_3]+[y_1,y_2,\omega(x_1, x_2, y_3)],\\
\nonumber &&\omega(x_1, x_2,\omega(y_1, y_2, y_3))\\
\label{eq:2-coc04}&=&\omega(\omega( x_1, x_2, y_1), y_2, y_3) + \omega(y_1, \omega( x_1, x_2, y_2), y_3)+ \omega(y_1, y_2, \omega( x_1, x_2, y_3)).
\end{eqnarray}

Since $M_\lambda$ is a $\g\dlam$-module, we get the equality
\begin{eqnarray*}
&&[x_1, x_2, [y_1, y_2, m]\dlam]\dlam \\
&=&[[ x_1, x_2, y_1]\dlam , y_2, m]\dlam  + [y_1, [ x_1, x_2, y_2]\dlam , m]\dlam  + [y_1, y_2, [ x_1, x_2, m]\dlam]\dlam,
\end{eqnarray*}
the left-hand side is equal to
\begin{eqnarray*}
&&[x_1, x_2, [y_1, y_2, m]+\lam\nu(y_1, y_2, m)]\dlam\\
&=&[x_1, x_2, [y_1, y_2, m]]+\lam\nu(x_1, x_2, [y_1, y_2, m])\\
&&+[x_1, x_2,\lam\nu(y_1, y_2, m)]+\lam\nu(x_1, x_2,\lam\nu(y_1, y_2, m))\\
&=&[x_1, x_2, [y_1, y_2, m]]+\lam\{\nu(x_1, x_2, [y_1, y_2, m])+[x_1, x_2,\nu(y_1, y_2, m)]\}\\
&&+\lam^2\nu(x_1, x_2,\nu(y_1, y_2, m)),
\end{eqnarray*}
and the right-hand side is equal to
\begin{eqnarray*}
&&[[x_1, x_2, y_1]+\lam\nu(x_1, x_2, y_1), y_2, m]\dlam  + [y_1, [ x_1, x_2, y_2]+\lam\nu(x_1, x_2, y_2), m]\dlam \\
&&+ [y_1, y_2, [ x_1, x_2, m]+\lam\nu(x_1, x_2, m)]\dlam\\
&=&[[ x_1, x_2, y_1] , y_2, m]  + [y_1, [ x_1, x_2, y_2] , m] + [y_1, y_2, [ x_1, x_2, m]]\\
&&+\lam\{\nu([x_1, x_2, y_1],y_2, m) +[\nu(x_1, x_2, y_1), y_2, m]\\
&&\qquad+\nu(y_1,[ x_1, x_2, y_2], m)+[y_1,\nu(x_1, x_2, y_2), m]\\
&&\qquad+\nu(y_1, y_2, [ x_1, x_2, m])+[y_1,y_2,\nu(x_1, x_2, m)] \}\\
&&+\lam^2\{\nu(\nu( x_1, x_2, y_1), y_2, m) + \nu(y_1, \omega( x_1, x_2, y_2), m) + \nu(y_1, y_2, \nu( x_1, x_2, m))\}.
\end{eqnarray*}
Thus, we have
\begin{eqnarray}
\nonumber &&\nu(x_1, x_2, [y_1, y_2, m])+[x_1, x_2,\nu(y_1, y_2, m)]\\
\nonumber &=&\nu([x_1, x_2, y_1],y_2, m)+\nu(y_1,[ x_1, x_2, y_2], m) +\nu(y_1, y_2, [ x_1, x_2, m])\\
\label{eq:2-coc05} &&+[\nu(x_1, x_2, y_1), y_2, m]+[y_1,\omega(x_1, x_2, y_2), m]+[y_1,y_2,\nu(x_1, x_2, m)],\\
\nonumber &&\nu(x_1, x_2,\nu(y_1, y_2, m))\\
\label{eq:2-coc06}&=&\nu(\omega( x_1, x_2, y_1), y_2, m) + \nu(y_1, \omega(x_1, x_2, y_2), m)+ \nu(y_1, y_2, \nu( x_1, x_2, m)).
\end{eqnarray}

By \eqref{eq:2-coc01},  \eqref{eq:2-coc03} and \eqref{eq:2-coc05}, we deduce that $(\theta,\omega,\nu)$ is a 2-cocycle of $(M,\g,f)$ with the coefficients in the adjoint representation.
Furthermore, by \eqref{eq:2-coc02}, \eqref{eq:2-coc04} and \eqref{eq:2-coc06},
$(M,\g,\theta,\omega,\nu)$ is a 3-Lie algebra in $\LM$.
\end{proof}

Next, we introduce the notion of Nijenhuis operators which gives trivial deformations.
In this case, we consider the second-order deformation $(M\dlam,\g\dlam,\dM\dlam)$ where
 \begin{eqnarray*}
\dM\dlam (m)&\triangleq&\dM (m)+\lambda\theta_1 (m),\\
{[x_1, x_2, x_3]}\dlam&\triangleq& [x_1, x_2, x_3]+\lambda\omega_1 (x_1, x_2, x_3)+\lambda^2\omega_2 (x_1, x_2, x_3),\\
{[x_1, x_2, m]}\dlam&\triangleq& [x_1, x_2, m]+ \lambda\nu_1(x_1, x_2, m)+ \lambda^2\nu_2(x_1, x_2, m).
 \end{eqnarray*}
%
%

\begin{defi}
A deformation is said to be trivial if there exist  linear maps $N_0:\frkg \to \frkg ,N_1:  M \to  M $,
such that $(T_0,T_1)$ is a homomorphism from $(M\dlam,\g\dlam,\dM\dlam)$ to $(M,\g,\dM) $, where $T_0 = \id + \lambda N_0$,
$T_1 = \id + \lambda N_1$.
\end{defi}

Note that $(T_0,T_1)$ is a homomorphism means that
\begin{eqnarray}
\label{nij00} \dM \circ T_1(m)&=&T_0\circ \dM\dlam(m),\\
\label{nij01}
T_{0} ([x_1,x_2,x_3]\dlam)&=&[T_{0} x_1,T_{0} x_2,T_{0} x_3],\\
\label{nij02}
T_1 ([x_1,x_2,{m}]\dlam)&=&[T_0x_1,T_0 x_2,T_1 {m}].
\end{eqnarray}

From \eqref{nij00}, we have that the left-hand is equal to
\begin{eqnarray*}
 \dM \circ T_1(m)&=&\dM (m)+\lambda \dM N_1(m),
\end{eqnarray*}
and the right-hand is equal to
\begin{eqnarray*}
&&T_0\circ \dM\dlam(m)\\
 &=&(\id + \lambda N_0)(\dM (m)+\lambda\theta (m))\\
 &=&\dM (m)+\lambda N_0(\dM (m))+\lambda \theta (m)+\lambda^2N_0\theta (m).
\end{eqnarray*}
Thus, we get
$$\theta (m)=(\dM N_1-N_0\dM) (m),\quad N_0\theta (m)=0.$$
It follows that $N$ must satisfy the following condition:
\begin{align}\label{Nijenhuis0}
N_0(\dM N_1-N_0\dM)=0.
\end{align}

%

From the equation \eqref{nij01} we get
\begin{eqnarray}
\notag&&\omega_1(x_1,x_2,x_3)+N_0[x_1,x_2,x_3]\\
\label{eq:Nijenhuis01}&&\qquad\qquad=[N_0(x_1),x_2,x_3]+[x_1,N_0(x_2),x_3]+[x_1,x_2,N_0(x_3)],\\
\notag&&\omega_2(x_1,x_2,x_3)+N_0\omega_1(x_1,x_2,x_3)\\
\label{eq:Nijenhuis02}&&\qquad\qquad=[N_0(x_1),N_0(x_2),x_3]+[N_0(x_1),x_2,N_0(x_3)]+[x_1,N_0(x_2),N_0(x_3)],\\
\label{eq:Nijenhuis03}&&N_0\omega_2(x_1,x_2,x_3)=[N_0(x_1),N_0(x_2),N_0(x_3)].
\end{eqnarray}

It follows from \eqref{eq:Nijenhuis01}, \eqref{eq:Nijenhuis02} and \eqref{eq:Nijenhuis03} that $N_0$ must satisfy the following condition
\begin{eqnarray}
\notag&&[N_0(x_1),N_0(x_2),N_0(x_3)]\\
\notag&=&N_0([N_0(x_1),N_0(x_2), x_3] + [N_0(x_1), x_2,N_0(x_3)] + [x_1,N_0(x_2),N_0(x_3)]) \\
\notag&&-N_0^2([N_0(x_1), x_2, x_3]+[x_1,N_0(x_2), x_3]+[x_1,x_2,N_0(x_3)])\\
\label{eq:Nijenhuis04}&&+N_0^3([x_1,x_2,x_3]).
\end{eqnarray}

By definition the left-hand side of equation \eqref{nij02} is equal to
\begin{eqnarray*}
T_\lam ([x_1,x_2,{m}]\dlam)&=&[x_1,x_2,{m}]+\lam(\nu_1(x_1,x_2,{m})+ N_1[x_1,x_2,{m}])\\
&&+\lam^2(\nu_2(x_1,x_2,{m})+N_1\nu_1(x_1,x_2,{m}))+\lam^3N_1\nu_2(x_1,x_2,{m}),
\end{eqnarray*}
and the right-hand side equals to
\begin{eqnarray*}
&&[T_\lam (x_1),T_\lam (x_2),T_\lam ({m})]\\
&=&[x_1+\lam N_0(x_1),x_2+\lam N_0(x_2),{m}+\lam N_1({m})]\\
&=&[x_1,x_2,{m}]+\lam([N_0(x_1),x_2,{m}]+[x_1,N_0(x_2),{m}]+[x_1,x_2,N_1({m})])\\
&&+\lam^2([N_0(x_1),N_0(x_2),{m}]+[N_0(x_1),x_2,N_1({m})]+[x_1,N_0(x_2),N_1({m})])\\
&&+\lam^3[N_0(x_1),N_0(x_2),N_1({m})].
\end{eqnarray*}
Thus, we have
\begin{eqnarray}
\notag&&\nu_1(x_1,x_2,{m})+N_1[x_1,x_2,{m}]\\
\label{eq:Nijenhuis1}&&\qquad\qquad=[N_0(x_1),x_2,{m}]+[x_1,N_0(x_2),{m}]+[x_1,x_2,N_1({m})],\\
\notag&&\nu_2(x_1,x_2,{m})+N_1\nu_1(x_1,x_2,{m})\\
\label{eq:Nijenhuis2}&&\qquad\qquad=[N_0(x_1),N_0(x_2),{m}]+[N_0(x_1),x_2,N_1({m})]+[x_1,N_0(x_2),N_1({m})],\\
\label{eq:Nijenhuis3}&&N_1\nu_2(x_1,x_2,{m})=[N_0(x_1),N_0(x_2),N_1({m})].
\end{eqnarray}

It follows from \eqref{eq:Nijenhuis1}, \eqref{eq:Nijenhuis2} and \eqref{eq:Nijenhuis3} that $(N_0,N_1)$ must satisfy the following condition
\begin{eqnarray}
\notag&&[N_0(x_1),N_0(x_2),N_1({m})]\\
\notag&=&N_1([N_0(x_1),N_0(x_2), {m}] + [N_0(x_1), x_2,N_1({m})] + [x_1,N_0(x_2),N_1({m})]) \\
\notag&&-N_1^2([N_0(x_1), x_2, {m}]+[x_1,N_0(x_2), {m}]+[x_1,x_2,N_1({m})])\\
\label{eq:Nijenhuis4}&&+N_1^3([x_1,x_2,{m}]).
\end{eqnarray}

 \begin{defi}\label{defi:Nijenhuis}
  A pair of linear maps $N=(N_0,N_1)$ is called a Nijenhuis operator if  the conditions  \eqref{Nijenhuis0}, \eqref{eq:Nijenhuis04} and \eqref{eq:Nijenhuis4} hold.
 \end{defi}

We have seen that any second-order trivial deformation produces a Nijenhuis operator.
Conversely, one proves that any Nijenhuis operator gives a second-order trivial deformation.
\begin{thm}\label{thm:Nijenhuis}
Let $N=(N_0,N_1)$ be a Nijenhuis operator. Then a  second-order deformation  can be obtained by putting
\begin{equation}\label{Nijenhuis}
\left\{\begin{array}{rll}
\theta (m) &=&(\dM N_1-N_0\dM) (m),\\
\omega_1(x_1,x_2,x_3)&=&[N_0(x_1),x_2,x_3]+[x_1,N_0(x_2),x_3]+[x_1,x_2,N_0(x_3)]-N_0[x_1,x_2,x_3],\\
\omega_2(x_1,x_2,x_3)&=&[N_0(x_1),N_0(x_2),x_3]+[N_0(x_1),x_2,N_0(x_3)]+[x_1,N_0(x_2),N_0(x_3)]\\
&&-N_1\omega_1(x_1,x_2,x_3),\\
\nu_1(x_1,x_2,{m})&=&[N_0(x_1),x_2,{m}]+[x_1,N_0(x_2),{m}]+[x_1,x_2,N_1({m})]-N_1[x_1,x_2,{m}],\\
\nu_2(x_1,x_2,{m})&=&[N_0(x_1),N_0(x_2),{m}]+[N_0(x_1),x_2,N_1({m})]+[x_1,N_0(x_2),N_1({m})]\\
&&-N_1\nu_1(x_1,x_2,{m}).
\end{array}\right.
\end{equation}
Furthermore, this deformation is trivial.
\end{thm}

\emptycomment{
\begin{proof}
Since $(\theta ,\omega ,\nu)=D(N_0,N_1)$, it is obvious that $(\theta ,\omega ,\nu)$ is a 2-cocycle.
It is easy to check that $(\theta ,\omega ,\nu)$ defines a 3-Lie algebra $(M,\g,\theta)$ in $\LM$ structure.
Thus by Theorem \ref{thm:deformation}, $(\theta ,\omega ,\nu)$ generates a deformation.
\end{proof}

\begin{eqnarray*}
\theta \nu(x_1, x_2, m)&=&(\dM N_1-N_0\dM)([N_0(x_1), x_2, m] + [x, N_1m] - N_1[x_1, x_2, m])\\
&=&\dM N_1([N_0(x_1), x_2, m] + [x, N_1m] - N_1[x_1, x_2, m])-N_0\dM([N_0(x_1), x_2, m] + [x, N_1m] - N_1[x_1, x_2, m])\\
&=&\dM N_1([N_0(x_1), x_2, m] + \dM N_1[x, N_1m] - \dM N_1 N_1[x_1, x_2, m])\\
&&-N_0[N_0x, \dM(m)] - N_0[x, \dM N_1m] + N_0 \dM N_1[x_1, x_2, m])\\
\end{eqnarray*}
\begin{eqnarray*}
\omega (x,\theta (m))&=&[N_0x, (\dM N_1-N_0\dM) (m)] + [x, N_0(\dM N_1-N_0\dM) (m)] - N_0[x, (\dM N_1-N_0\dM) (m)]\\
&=&[N_0x, \dM N_1 (m)]-[N_0x,N_0\dM (m)] - N_0[x, \dM N_1 (m)]+ N_0[x, N_0\dM (m)]\\
\end{eqnarray*}
}

Finally, we consider the formal deformations of any order.
Let $(M,\g, f)$ be a 3-Lie algebra in $\LM$ and $\theta_i: M \to\g,~\omega_i:\wedge^3\g\to \g ,~\nu_i: \wedge^2\g\otimes M \to M, i\geqslant 0$ be linear maps.
Consider the $\lambda$-parametrized family of linear operations:
\begin{eqnarray*}
\dM\dlam (m)&\triangleq&\dM (m)+\lambda\theta_1 (m)+\lambda^2\theta_1 (m)+\dots,\\
{[x_1, x_2, x_3]}\dlam&\triangleq& [x_1, x_2, x_3]+\lambda\omega_1 (x_1, x_2, x_3)+\lambda^2\omega_2 (x_1, x_2, x_3)+\dots,\\
{[x_1, x_2, m]}\dlam&\triangleq& [x_1, x_2, m]+ \lambda\nu_1(x_1, x_2, m)+ \lambda^2\nu_2(x_1, x_2, m)+\dots.
 \end{eqnarray*}
In order that  $(M\dlam, \g\dlam, f\dlam)$ forms a 3-Lie algebra in $\LM$, we must have
\begin{eqnarray*}
\omega\dlam (x_1, x_2, \omega\dlam(y_1,y_2,y_3))
&=&\omega\dlam (\omega\dlam(x_1, x_2, y_1),y_2,y_3)+\omega\dlam (y_1,\omega\dlam(x_1, x_2, y_2), y_3)\\
&&+\omega\dlam (y_1,y_2,\omega\dlam(x_1, x_2,y_3)),\\
\nu\dlam (x_1, x_2, \nu\dlam(y_1,y_2,m))
&=&\nu\dlam (\omega\dlam(x_1, x_2, y_1),y_2,m)+\nu\dlam (y_1,\omega\dlam(x_1, x_2,y_2), m)\\
&&+\nu\dlam (y_1,y_2,\nu\dlam(x_1, x_2, m)),\\
f\dlam\nu\dlam (x_1, x_2, m)&=&\omega\dlam(x_1, x_2,f\dlam(m)).
\end{eqnarray*}
which implies that
\begin{eqnarray}
\notag\sum_{i+j=k}\omega_i (x_1, x_2, \omega_j(y_1,y_2,y_3))
&=&\sum_{i+j=k}\omega_i (\omega_j(x_1, x_2, y_1),y_2,y_3)+\omega_i (y_1,\omega_j(x_1, x_2, y_2), y_3)\\
\label{eq:formal1}&&+\omega_i (y_1,y_2,\omega_j(x_1, x_2, y_3)),\\
\notag\sum_{i+j=k}\nu_i (x_1, x_2, \nu_j(y_1,y_2,m))
&=&\sum_{i+j=k}\nu_i (\omega_j(x_1, x_2, y_1),y_2,m)+\omega_i (y_1,\nu_j(x_1, x_2,m), y_3)\\
\label{eq:formal2}&&+\nu_i (y_1,y_2,\nu_j(x_1, x_2, m)),\\
\label{eq:formal3}\sum_{i+j=k}f_i\nu_j(x_1, x_2, m)&=&\sum_{i+j=k}\omega_i(x_1, x_2,f_j(m)),
\end{eqnarray}
 where we denote $f_0(m)=\dM (m), \omega_0(x_1, x_2, x_3)=[x_1, x_2, x_3]$, $\nu_0(x_1, x_2, m)=[x_1, x_2, m]$.

For $k=0$, conditions \eqref{eq:formal1}, \eqref{eq:formal2} and \eqref{eq:formal3} are equivalent to $(\omega_0,\nu_0, f_0)$ is an $n$-Lie algebra in $\LM$.
For $k=1$, these conditions  are equivalent to
\begin{eqnarray*}
&&\omega_1(x_1, x_2, [y_1, y_2, y_3])+[x_1, x_2,\omega_1(y_1, y_2, y_3)]\\
&=&\omega_1([x_1, x_2, y_1],y_2, y_3)+\omega_1(y_1,[ x_1, x_2, y_2], y_3) +\omega_1(y_1, y_2, [ x_1, x_2, y_3])\\
&&+[\omega_1(x_1, x_2, y_1), y_2, y_3]+[y_1,\omega_1(x_1, x_2, y_2), y_3]+[y_1,y_2,\omega_1(x_1, x_2, y_3)],
\end{eqnarray*}
\begin{eqnarray*}
 &&\nu_1(x_1, x_2, [y_1, y_2, m])+[x_1, x_2,\nu_1(y_1, y_2, m)]\\
 &=&\nu_1([x_1, x_2, y_1],y_2, m)+\nu_1(y_1,[ x_1, x_2, y_2], m) +\nu_1(y_1, y_2, [ x_1, x_2, m])\\
&&+[\nu_1(x_1, x_2, y_1), y_2, m]+[y_1,\omega_1(x_1, x_2, y_2), m]+[y_1,y_2,\nu_1(x_1, x_2, m)],
\end{eqnarray*}
\begin{eqnarray*}
&&\theta_1 ([x_1, x_2, m])+f \nu_1 (x_1, x_2, m)=\omega_1 (x_1, x_2, f(m))+[x_1, x_2,\theta_1 (m)].
\end{eqnarray*}
Thus for $k=1$, $(\omega_1,\nu_1,\theta_1)\in C^{2}((M,\g,f),(M,\g,f))$ is a 2-cocycle.

\begin{defi}
The 2-cochain $(\omega_1,\nu_1,\theta_1)$ is called the infinitesimal of the deformation
$(\omega\dlam,\nu\dlam,f\dlam)$. More generally, if $(\omega_i,\nu_i,f_i)=0$ for $1\leqslant i\leqslant (n -1)$,
and $(\omega_n,\nu_n,f_n)$ is a non-zero cochain in $C^{2}((M,\g,f),(M,\g,f))$, then $(\omega_n,\nu_n,f_n)$ is called the $n$-infinitesimal of
the deformation $(\omega\dlam,\nu\dlam,f\dlam)$.
\end{defi}

Let $(\omega\dlam,\nu\dlam,f\dlam)$ and $(\omega'\dlam,\nu'\dlam,f'\dlam)$ be two deformation.
We say that they are equivalent if there exists a formal isomorphism
$(\Phi\dlam,\Psi\dlam):(M'\dlam, \g'\dlam, f'\dlam)\to (M\dlam, \g\dlam, f\dlam)$ such that
$\omega'\dlam(x_1, x_2, x_3)=\Psi^{-1}\dlam\omega\dlam(\Psi\dlam(x_1),\Psi\dlam(x_2),\Psi\dlam(x_3))$.

A deformation $(\omega\dlam,\nu\dlam,f\dlam)$ is said to be the trivial deformation if it is  equivalent to $(\omega_0,\nu_0, f)$.

\begin{thm}
Let $(\omega\dlam,\nu\dlam,f\dlam)$ and $(\omega'\dlam,\nu'\dlam,f'\dlam)$ be equivalent deformations of  $(M,\g, f)$ be a 3-Lie algebra in $\LM$,
then the first-order terms of them belong to the same cohomology class in the second cohomology group $H^2((M,\g, f),(M,\g, f))$.
\end{thm}

\pf
Let $(\Phi\dlam,\Psi\dlam):(M\dlam, \g\dlam, f\dlam)\to (M'\dlam, \g'\dlam, f'\dlam)$ be an equivalence
where $\Phi\dlam=\id_M+\lambda\phi_1+\lambda^2\phi_2+\dots$ and $\Psi\dlam=\id_M+\lambda\psi_1+\lambda^2\psi_2+\dots$.
Then we have
$\Psi\dlam\omega'\dlam(x_1, x_2, x_3)=\omega\dlam(\Psi\dlam(x_1),\Psi\dlam(x_2),\Psi\dlam(x_3))$
$\Psi\dlam\nu'\dlam(x_1, x_2, m)=\nu\dlam(\Phi\dlam(x_1),\Phi\dlam(x_2),\Psi\dlam(m))$
\qed

A 3-Lie algebra $(M,\g, f)$ in $\LM$ is called rigid if every deformation $(\omega\dlam,\nu\dlam,f\dlam)$ is equivalent to the trivial deformation.

\begin{thm}
If $H^2((M,\g, f),(M,\g, f))=0$, then $(M,\g, f)$ is rigid.
\end{thm}

\begin{proof}
Let $(\omega\dlam,\nu\dlam,f\dlam)$ be a  deformation of $(M,\g, f)$.
It follows from the above results that  $D(\omega\dlam,\nu\dlam,f\dlam)=0$, that is $(\omega\dlam,\nu\dlam,f\dlam)\in Z^2((M,\g, f),(M,\g, f))$.
Now assume $H^2((M,\g, f))=0$, we can find $(N_0,N_1)$ such that $(\omega\dlam,\nu\dlam,f\dlam)=D(N_0,N_1)$.
The proof is completed.
\end{proof}

\section*{Acknowledgements}
This research was supported by the National Natural Science Foundation of China (No. 11961049).

\section*{Data availibility}
Data sharing not applicable to this article as no datasets were generated or analysed during the current study.

\section*{Declarations}
\textbf{Conflict of interest}
On behalf of all authors, the corresponding author states that there is no conflict of interest.

\end{document}